\documentclass[11pt,reqno]{amsart}

\usepackage{graphicx}
\usepackage{mathrsfs}
\usepackage{color}
\usepackage{comment}
\usepackage{hyperref}
\usepackage{comment}
\usepackage{amssymb}

\numberwithin{equation}{section}

\newtheorem{prop}{Proposition}[section]
\newtheorem{theo}[prop]{Theorem}
\newtheorem*{theo*}{Theorem}
\newtheorem{lemm}[prop]{Lemma}
\newtheorem{coro}[prop]{Corollary}

\newtheorem*{claim}{Claim}
\newtheorem{defi}[prop]{Definition}
\newtheorem{rema}[prop]{Remark}

\theoremstyle{definition}

\newtheorem{quest}[prop]{Question}
\newtheorem{conj}[prop]{Conjecture}

\newcommand{\CC}{\mathbf{C}}

\newcommand{\RR}{\mathbf{R}}
\renewcommand{\SS}{\mathbf{S}}
\newcommand{\TT}{\mathbf{T}}
\newcommand{\ZZ}{\mathbf{Z}}

\newcommand{\cH}{\mathcal H}

\newcommand{\cM}{\mathcal M}

\newcommand{\cS}{\mathcal S}

\newcommand{\sC}{\mathscr{C}}

\DeclareMathOperator{\Tr}{Tr}

\DeclareMathOperator{\codim}{codim}

\DeclareMathOperator{\spt}{spt}

\DeclareMathOperator{\loc}{loc}

\DeclareMathOperator{\Vol}{Vol}

\DeclareMathOperator{\Area}{Area}

\DeclareMathOperator{\dist}{dist}

\DeclareMathOperator{\Ric}{Ric}

\DeclareMathOperator{\sing}{sing}
\DeclareMathOperator{\reg}{reg}

\newcommand{\td}[2]{\frac{d #1}{d #2}}
\newcommand{\bangle}[1]{\left\langle #1 \right\rangle}

\newcommand{\ep}{\varepsilon}

\numberwithin{equation}{section}

\begin{document}

\title{Positive scalar curvature with skeleton singularities}

\author{Chao Li}
\address{Department of Mathematics, Stanford University}
\email{rchlch@stanford.edu}
\author{Christos Mantoulidis}
\address{Department of Mathematics, Massachusetts Institute of Technology}
\email{c.mantoulidis@mit.edu}

\begin{abstract}
    We study positive scalar curvature on the regular part of Riemannian manifolds with singular, uniformly Euclidean ($L^\infty$) metrics that consolidate Gromov's scalar curvature polyhedral comparison theory and edge metrics that appear in the study of Einstein manifolds. We show that, in all dimensions, edge singularities with cone angles $\leq 2\pi$ along codimension-2 submanifolds do not affect the Yamabe type. In three dimensions, we prove the same for more general singular sets, which are allowed to stratify along 1-skeletons, exhibiting edge singularities (angles $\leq 2\pi$) and arbitrary $L^\infty$ isolated point singularities. We derive, as an application of our techniques, Positive Mass Theorems for asymptotically flat manifolds with analogous singularities.
\end{abstract}

\maketitle

\section{Introduction}

\subsection{Background and statement of results}

Comparison, rigidity, and compactness theorems, and their  applications on the study of low-regularity Riemannian metrics satisfying some ``weak'' curvature conditions, has been a central theme in Riemannian geometry. There has been considerable success here in the case of \emph{sectional} curvature lower bounds (Alexandrov spaces; see, e.g., \cite{AleksandrovRerestovskiiNikolaev86, BuragoBuragoIvanov01}) as well as \emph{Ricci} curvature lower bounds (Cheeger-Colding-Naber theory; see, e.g., \cite{CheegerColding97, CheegerColding00a, CheegerColding00b, ColdingNaber12, CheegerNaber13}; for an optimal transport approach, see, e.g., \cite{LottVillani09, Sturm06a, Sturm06b, Sturm06c}).

The case of scalar curvature lower bounds is not as well understood, possibly due to a lack of a relevant geometric comparison theory. Gromov \cite{Gromov14} recently proposed a polyhedral comparison theory for the study of positive scalar curvature\footnote{which was recently confirmed by the first-named author, alongside a rigidity result conjectured by Gromov, using altogether different methods; see \cite{Li}.}, building on \eqref{eq:sigma.smooth.obstruction} of 
Theorem \ref{theo:sigma.obstruction.and.rigidity}, which characterizes smooth manifolds as being in one of three families:
\begin{enumerate}
    \item those that carry (smooth) metrics with positive scalar curvature;
    \item those that don't, but do carry metrics with nonnegative scalar curvature and are automatically Ricci-flat;
    \item those that carry neither.
\end{enumerate}
(Unless otherwise specified, all manifolds in this paper are 
smooth and of dimension at least three.) The three families are conveniently distinguished by the sign of the smooth $\sigma$-invariant (or Schoen invariant) of $M$:
\[ \sigma(M) := \sup \{ \mathcal{Y}(M, [g_0]) : [g_0] \text{ is a conformal class of metrics on M} \}, \]
where
\[ \mathcal{Y}(M, [g_0]) := \inf \left\{ \int_M R(g) \, d\Vol_g : g \in [g_0], \Vol_g(M) = 1 \right\} \]
and $R(g)$ denotes the scalar curvature of a smooth metric $g$. The sign of $\sigma(M)$ determines the so-called Yamabe type of $M$: positive, zero, or negative. For example: $\sigma(\SS^n) > 0$, $\sigma(\TT^n) = 0$, and $\sigma(\Sigma_g \times (\SS^1)^{n-2}) < 0$ for any surface $\Sigma_g$ with genus $g \geq 2$. (See \cite[Corollary A]{GromovLawson84}.)

\begin{theo}[See {\cite{KazdanWarner75, Schoen89}}] \label{theo:sigma.obstruction.and.rigidity}
    Let $M^n$, $n \geq 3$, be closed. Then
    \begin{equation} \label{eq:sigma.smooth.obstruction}
        \sigma(M) > 0 \iff M \text{ carries a smooth metric } g \text{ with } R(g) > 0.
    \end{equation}
    Moreover,
    \begin{equation} \label{eq:sigma.smooth.rigidity}
        R(g) \geq 0 \text{ and } \sigma(M) \leq 0 \implies \Ric(g) \equiv 0.
    \end{equation}
\end{theo}

Gromov's \cite{Gromov14} 
polyhedral comparison theory for positive scalar curvature 
relies on 
a construction 
of a 3-torus with a \emph{singular} uniformly Euclidean metric with positive scalar curvature. (See Section \ref{subsec:gromov.comparison}.) Motivated by further implications of such a study of weak notions of positive scalar curvature, one naturally wonders which uniformly Euclidean metrics with nonnegative scalar curvature (on their regular part) are compatible with the global obstruction and rigidity aspects of Theorem \ref{theo:sigma.obstruction.and.rigidity}:

\begin{quest}[Weakly nonnegative scalar curvature, globally] \label{quest:singular.psc.rigidity}
	Suppose $g$ is an $L^\infty$ metric on $M$ that is smooth away from a compact subset $\cS \subset M$. What conditions on $\cS$, $g$, ensure that
	\begin{multline} \label{eq:sigma.singular.rigidity}
	    R(g) \geq 0 \text{ on } M \setminus \cS \text{ and } \sigma(M) \leq 0 \\
	    \implies g \text{ extends smoothly to } M \text{ and } \Ric(g) \equiv 0?
	\end{multline}
	Said otherwise, under what conditions on $\cS$, $g$ does $M$ carry \emph{singular} metrics with nonnegative scalar curvature (on the regular part), but no such \emph{smooth} metrics?
\end{quest}

\begin{defi}[Uniformly Euclidean ($L^\infty$) metrics] \label{defi:linfty.metrics}
    We define the class of $L^\infty$ metrics on a closed manifold $M$ to consist of all measurable sections of  $\operatorname{Sym}_2(T^* M)$ such that
    \[ \Lambda^{-1} g_0 \leq g \leq \Lambda g_0 \text{ a.e. on } M \]
    for some smooth metric $g_0$ on $M$ and some $\Lambda > 0$.
\end{defi}

Let us discuss what is known in the general direction of Question \ref{quest:singular.psc.rigidity}. We also survey previous (singular) Positive Mass Theorem results because, in the smooth setting, \cite[Section 6]{Lohkamp99},  \cite[Proposition 5.4]{SchoenYau17}, and  
Theorem \ref{theo:sigma.obstruction.and.rigidity} \emph{imply} the Positive Mass Theorem of Schoen-Yau \cite{SchoenYau79pmt} and Witten \cite{Witten81} for complete asymptotically flat $(M^n, g)$ with $R(g) \geq 0$. 

\vskip 1em

\begin{center}
{\bf Codimension 1.}
\end{center}

The case that is best understood is
        \[ \codim (\cS \subset M) = 1, \]
        where $\cS$ is a closed embedded hypersurface with trivial normal bundle and where the ambient metric $g$ induces the same smooth metric on $\cS$ from both sides.
        
        One \emph{cannot} hope for Question \ref{quest:singular.psc.rigidity} to be valid in such generality. To maintain any hope of validity, one must make an additional \emph{geometric} assumption: the sum of mean curvatures of $\cS$ computed with respect to the two unit normals as outward unit normals has to be nonnegative. See Section \ref{subsec:codimension.1.appendix} for more details.
        
        There have been three approaches, all subject to the geometric assumption just described. The first, and closest in spirit to this paper, is to combine the conformal method with arbitrarily fine desingularizations that are aware of the ambient geometry; this was first carried out in the positive mass setting by Miao \cite{Miao02}; see also \cite{Corvino00, Bray01}. The second is to use Ricci flow as a smoothing tool; see \cite{ShiTam16}, or \cite{McFeronSzekelyhidi12} for the positive mass analog. The third is to use spinors; see \cite{ShiTam02, LeeLeFloch15} for positive mass, or \cite{ChruscielHerzlich03} for positive mass in asymptotically hyperbolic spaces.
        
\vskip 1em
\begin{center}
{\bf Codimension 2.}
\end{center}

Much less is known when $\cS \subset M$ is a closed embedded submanifold with
        \[ \codim(\cS \subset M) \geq 2. \]
        Nonetheless, one can \emph{still} not expect Question \ref{quest:singular.psc.rigidity} to be valid in such generality and needs to decide on additional geometric assumptions; see, e.g., Section \ref{subsec:codimension.2.appendix} for counterexamples.
        
        One approach, which we won't pursue, is to strengthen the regularity assumptions on $g$; to that end, Shi-Tam \cite{ShiTam16} proved (using Ricci flow) that \eqref{eq:sigma.singular.rigidity} is \emph{true} if $g$ is Lipschitz across $\cS$. See \cite{Lee13} for a result in the positive mass setting that uses the conformal method.

        In our work in codimension two, we opt to keep the low ($L^\infty$) regularity assumption and instead study metrics whose singularities are of ``edge'' type (see Definition \ref{defi:edge.singularity}), which consolidate Gromov's polyhedral comparison theory together with the study of singularities in Einstein manifolds. Edge singularities have been studied intensively recently due to the Yau-Tian-Donaldson program in K\"ahler-Einstein geometry; see, e.g., \cite{ChenDonaldsonSun15a, ChenDonaldsonSun15b, ChenDonaldsonSun15c, Tian15, JeffresMazzeoRubinstein16}, or \cite{AtiyahLebrun13} for non-complex-geometric results in (real) dimension four. See Section \ref{subsec:source.edge.metrics} for examples of edge metrics.

        Our first theorem deals with codimension two edge singularities in all dimensions, $n \geq 3$:

\begin{theo} \label{theo:main.highD}
	Let $M^n$ be closed, with $\sigma(M) \leq 0$, and $g$ a metric such that:
	\begin{enumerate}
	    \item $g \in L^\infty(M) \cap C^2_{\loc}(M \setminus \cS)$, $\cS \subset M$ is a codimension-2 closed submanifold, and $g$ is an $\eta$-regular edge metric along $\cS$ with $\eta > 2 - \tfrac{4}{n}$ and cone angles $0 < 2\pi(\beta+1) \leq 2\pi$,
	    \item $R(g) \geq 0$ on $M \setminus \cS$.
	\end{enumerate}
	Then $g$ extends to a smooth Ricci-flat metric everywhere on $M$.
\end{theo}

We note, in Section \ref{subsec:codimension.2.appendix}, that Theorem \ref{theo:main.highD} would be \emph{false} if one were to allow edge metrics with cone angles $> 2\pi$.

Despite recurring success in the study of Einstein metrics, the role of edge metrics in scalar curvature geometry has not been understood with depth. We expect general stratified singular sets with edge singularities along the codimension two strata to appear in the study of singular scalar curvature in a natural way. See, e.g.,   Akutagawa-Carron-Mazzeo \cite{AkutagawaCarronMazzeo14} for the singular Yamabe problem in this setting.

\vskip 1em
\begin{center}
{\bf Codimension 3.}
\end{center}

Rick Schoen has conjectured that the situation is drastically different in codimension three than in codimensions one or two: one shouldn't need any additional regularity assumptions beyond $L^\infty$ for \eqref{eq:sigma.singular.rigidity} to hold true:

\begin{conj} \label{conj:schoen}
    Suppose $g$ is an $L^\infty$ metric on $M$ that is smooth away from a closed, embedded submanifold $\cS \subset M$ with $\codim (\cS \subset M) \geq 3$. Then
	\begin{multline*}
	    R(g) \geq 0 \text{ on } M \setminus \cS \text{ and } \sigma(M) \leq 0 \\
	    \implies g \text{ extends smoothly to } M \text{ and } \Ric(g) \equiv 0.
	\end{multline*}
\end{conj}

We confirm Conjecture \ref{conj:schoen}, when $n \, (= \dim M) = 3$, as a corollary to our second theorem.

\begin{coro} \label{coro:main.3D}
	Let $M^3$ be closed, with $\sigma(M) \leq 0$. If $\cS \subset M$ is a finite set, $g$ is an $L^\infty(M) \cap C^{2,\alpha}_{\loc}(M \setminus \cS)$ metric, $\alpha \in (0,1)$, and $R(g) \geq 0$ on $M \setminus \cS$, then $g$ is a smooth flat metric everywhere on $M$.
\end{coro}

Our second theorem is specific to the three-dimensional case, where we allow \emph{stratified} singular sets of codimension two. We prove (see Definitions \ref{defi:edge.singularity}, \ref{defi:nondegenerate.arrangement}):

\begin{theo} \label{theo:main.3D}
	Let $M^3$ be closed, with $\sigma(M) \leq 0$, and $g$ a metric such that:
	\begin{enumerate}
	    \item $g \in L^\infty(M) \cap C^{2,\alpha}_{\loc}(M \setminus \cS)$, $\alpha \in (0,1)$, where $\cS \subset M$ is a compact nondegenerate 1-skeleton, and is an $\eta$-regular edge metric along $\reg \cS$ with $\eta > \tfrac{2}{3}$ and cone angles $0 < \delta \leq 2\pi(\beta+1) \leq 2\pi$,
	    \item $R(g) \geq 0$ on $M \setminus \cS$.
	\end{enumerate}
	Then $g$ extends to a smooth flat metric everywhere on $M$.
\end{theo}

\subsection{Applications to asymptotically flat manifolds}

Recall that, in the smooth setting,  \cite[Section 6]{Lohkamp99} and \cite[Proposition 5.4]{SchoenYau17} together imply that the positive mass theorem follows from Theorem \ref{theo:sigma.obstruction.and.rigidity}. The constructive techniques in our proofs of Theorems \ref{theo:main.highD}, \ref{theo:main.3D} similarly allow us to obtain Riemannian Positive Mass Theorems for singular metrics on analogous manifolds:

\begin{theo} \label{theo:PMT.highD}
	Let $(M^n,g)$ be a complete asymptotically flat manifold, such that:
	\begin{enumerate}
	    \item $g \in L^\infty(M) \cap C^2_{\loc}(M \setminus \cS)$, $\cS \subset M \setminus \partial M$ is a closed codimension two submanifold, and $g$ is an $\eta$-regular edge metric along $\cS$ with $\eta > 2 - \tfrac{4}{n}$ and cone angles $0 < 2\pi(\beta+1) \leq 2\pi$,
	    \item $\partial M = \emptyset$, or its mean curvature vectors vanish or point inside $M$,
	    \item $R(g) \geq 0$ on $M \setminus \cS$.
	\end{enumerate}
	Then the ADM mass of each end of $M$ is nonnegative. Moreover, if the mass of any end is zero, then $(M^n, g) \cong (\RR^n, \delta)$.
\end{theo}

\begin{theo}\label{theo:PMT.3D}
    Let $(M^3,g)$ be a complete asymptotically flat three-manifold, such that:
	\begin{enumerate}
	    \item $g \in L^\infty(M) \cap C^{2,\alpha}_{\loc}(M \setminus \cS)$, $\alpha \in (0,1)$, with $\cS \subset M \setminus \partial M$ a compact nondegenerate 1-skeleton, so that $g$ is an $\eta$-regular edge metric along $\reg \cS$ with $\eta > \tfrac{2}{3}$ and cone angles $0 < \delta \leq 2\pi(\beta+1) \leq 2\pi$,
	    \item $\partial M = \emptyset$, or its mean curvature vectors vanish or point inside $M$,
	    \item $R(g) \geq 0$ on $M \setminus \cS$,
	\end{enumerate}
	Then the ADM mass each end of $M$ is nonnegative. Moreover, if the mass of any end is zero, then $(M, g) \cong (\RR^3, \delta)$.
\end{theo}

\begin{rema}
    Our proofs of Theorems \ref{theo:main.highD}, \ref{theo:main.3D}, \ref{theo:PMT.highD}, \ref{theo:PMT.3D} make use of fine desingularizations in the spirit of Miao \cite{Miao02}. This constructive approach has been pursued in part for reasons of compatibility with the Sormani-Wenger \cite{SormaniWenger11} notion of ``intrinsic flat'' distance between Riemannian manifolds, which (see  \cite{Sormani}) work of Gromov \cite{Gromov14} suggests is the ``correct'' notion for taking limits of manifolds with lower scalar curvature bounds. See Section \ref{subsec:intrinsic.flat.distance} for more discussion.
\end{rema}

\vskip 1em

{\bf Acknowledgments.} The authors would like to thank Rick Schoen, Brian White, Rafe Mazzeo, Pengzi Miao, and Or Hershkovits for stimulating conversations on the subject of this paper, as well as Gerhard Huisken, Dan Lee, Andr\'e Neves, Yuguang Shi, and Peter Topping for their interest in this work. The first author would like to thank ETH-FIM for their hospitality, during which part of this work was carried out. The second author would like to thank the Ric Weiland Graduate Fellowship at Stanford, which partially supported the early portion of this research.

\section{Edge singularities} \label{sec:edge.singularities}

The starting point of our discussion is the classical example of \emph{isolated conical singularities} on two-dimensional Riemannian manifolds. 

Assume $M$ is a closed Riemann surface, $\{ p_1, \ldots, p_k \} \subset M$, and $g$ is an  $L^\infty(M) \cap C^2_{\loc}(M \setminus \{p_1, \ldots, p_k\})$ metric. We call $p_i$, $i = 1, \ldots, k$, an isolated conical singularity with cone angle $2\pi(\beta_i+1)$, $\beta_i \in (-1, \infty)$, if around $p_i$ there exist coordinates so that
\begin{equation} \label{eq:cone.singularity.euclidean}
    g = dr^2 + (\beta_i + 1)^2 r^2 d\theta^2.
\end{equation}
See Figure \ref{fig:cone.smoothing} for a graphical illustration of a model isolated conical singularity.

\begin{rema} \label{rema:cone.singularity.complex}
    In complex geometry one often works with the complex variable
    \[ z = [(\beta_i+1)r]^{1/(\beta_i+1)} e^{\sqrt{-1} \theta} \in \CC \setminus \{0\}, \]
    asserting that $g =  |z|^{2\beta_i} |dz|^2$, $z \neq 0$. We will not pursue this here.
\end{rema}

The Gauss-Bonnet formula in this setting of isolated conical singularities is
\begin{equation} \label{eq:2d.gauss.bonnet}
    \int_{M\setminus\{p_1, \ldots, p_k\}} K_g \, d\Area_g - 2\pi \sum_{i=1}^k \beta_i =2\pi \chi(M).
\end{equation}
This can be seen, for instance, by excising arbitrarily small disks around the conical points and taking limits. (See also Lemma \ref{lemm:smooth.out.edge.singularities} below.)

As a straightforward corollary of \eqref{eq:2d.gauss.bonnet}, the presence of conical singularities all of whose cone angles are $\leq 2\pi$ does \emph{not} affect the Yamabe type of $M$. On the other hand, conical singularities with cone angle bigger than $2\pi$ can affect the Yamabe type in the negative. We give an example in Section \ref{subsec:codimension.2.appendix}.

Let's proceed to the more interesting higher dimensional analog. A natural extension of the previous situation to higher dimensions leads to the definition of an edge singularity. Qualitatively, the singular metric $g$ may be viewed as a family of two-dimensional conical metrics along a smooth $(n-2)$-dimensional submanifold.

\begin{defi}[Edge singularities] \label{defi:edge.singularity}
	Let $N^{n-2} \subset M^n$ be a codimension-2 submanifold (without boundary). We call $g$ an $\eta$-regular edge metric along $N$ with data $(\eta, \beta, \sigma, \omega, \varrho, h)$, where $\eta \in (0, \infty)$, $\beta : N \to (-1, \infty)$ is $C^2$, $\sigma$ is a $C^2$ 1-form on $N$, $\omega$ is a $C^2$ metric on $N$, $\varrho : N \to (0, \infty)$ is $C^2$ on $N$, $h$ is a $C^2$ symmetric 2-tensor on $U$, if for some open set $U \supseteq N$,
	\begin{equation} \label{eq:edge.singularity}
	    g = dr^2 + (\beta+1)^2 r^2 (d\theta + \sigma)^2 + \omega + r^{1+\eta} h \text{ on } U \setminus N,
	\end{equation}
	\begin{equation} \label{eq:edge.singularity.neighborhood}
	    \{ (r, \theta, y) : r < \varrho(y), \theta \in \SS^1, y \in N \} \subseteq U,
	\end{equation}
	and
	\begin{multline} \label{eq:edge.singularity.structural}
	    \Vert \beta \Vert_{C^2(N)} + \Vert \sigma \Vert_{C^2(N)} + \Vert \omega \Vert_{C^2(N)} + \Vert (\det \omega)^{-1} \Vert_{C^0(N)} \\
	    + \Vert \varrho \Vert_{C^1(N)} + \Vert \varrho^{-1-\eta} \partial^2 \varrho \Vert_{C^0(N)} + \Vert h \Vert_{C^2(U)} < \infty.
	\end{multline}
	Specifically, we require that $U$ can be covered with Euclidean local coordinate charts $(x^1, x^2, y^1, \ldots, y^{n-2})$, where $r e^{\sqrt{-1} \theta} = x^1 + \sqrt{-1} x^2$ and $(y^1, \ldots, y^n) \in N$, in which
	\[ \Vert \beta \Vert_{L^\infty(N)} + \Vert \partial_i \beta \Vert_{C^0(N)} + \Vert \partial_{i} \partial_j \beta \Vert_{L^\infty(N)} < \infty, \]
	\[ \Vert \sigma_i \Vert_{L^\infty(N)} + \Vert \partial_i \sigma_j \Vert_{L^\infty(N)} + \Vert \partial_{i} \partial_j \sigma_k \Vert_{L^\infty(N)} < \infty, \]
	\[ \Vert (\det \omega_{ij})^{-1} \Vert_{L^\infty(N)} + \Vert \omega_{ij} \Vert_{L^\infty(N)} + \Vert \partial_i \omega_{jk} \Vert_{L^\infty(N)} + \Vert \partial_{i} \partial_j \omega_{k\ell} \Vert_{L^\infty(N)} < \infty, \]
	\[ \Vert h_{\alpha \beta} \Vert_{L^\infty(U)} + \Vert \partial_\alpha \omega_{\beta \gamma} \Vert_{L^\infty(Y)} + \Vert \partial_{\alpha} \partial_\beta \omega_{\gamma \delta} \Vert_{L^\infty(U)} < \infty, \]
	\[ \Vert \varrho \Vert_{L^\infty} + \Vert \partial_i \varrho \Vert_{L^\infty(N)} + \Vert \varrho^{-\eta} \partial_i \partial_j \varrho \Vert_{L^\infty(N)} < \infty. \]
	Latin indices only run through $(y^1, \ldots, y^{n-2})$ on $N$, while Greek indices run through all coordinates $(x^1, x^2, y^1, \ldots, y^{n-2})$ on $U$. 
\end{defi}

This definition is taken from  \cite[(1.1)-(1.2)]{AtiyahLebrun13}, and has corresponding analogs in K\"ahler-Einstein geometry. The $\varrho$ structural requirement did not appear in \cite{AtiyahLebrun13}, which only considered compact manifolds, but it is needed here for our general smoothing procedure in case $N$ is noncompact. (Notice that the $\varrho$-requirement is trivially true when $N$ is compact.) It is a mild requirement that stipulates that our domain of validity of the cone expansion does not degenerate too wildly near the endpoints.

We conclude our collection of definitions with the notion of skeletons:

\begin{defi}[Skeletons] \label{defi:nondegenerate.arrangement}
	We say that a compact subset $\cS \subset M$ is an $(n-2)$-skeleton if $\cS = N_1 \cup \cdots \cup N_k$, where $N_1, \ldots, N_k \subset M$ are compact submanifolds-with-boundary (possibly empty), each with dimension $\leq n-2$, and which are such that $N_\ell \cap N_\ell' \subset \partial N_\ell \cup \partial N_\ell'$ for all $\ell, \ell'$. We denote
	\begin{multline*}
		\reg \cS := \bigcup \Big\{ \cS \cap W : W \subset U \text{ is open and } \cS \cap W \text{ is a smooth} \\
			(n-2)-\text{dimensional submanifold (without boundary)} \Big\},
	\end{multline*}
	and $\sing \cS := \cS \setminus \reg \cS$. A skeleton $\cS$ is said to be \emph{nondegenerate} if there are no two inner-pointing conormals of $\partial N_\ell \subset N_\ell$ and $\partial N_\ell \subset N_\ell$ ($\ell \neq \ell'$) that coincide.
\end{defi}

One could ostensibly also want to allow higher stratum singularities (i.e., codimension-1) away from $\cS$ (e.g., in the spirit of Miao \cite{Miao02}, \cite{ShiTam16}). We do not pursue this direction in this paper.

\section{Smoothing edge singularities, I} \label{sec:smoothing.edge.singularities}

We will prove the following smoothing lemma.

\begin{lemm} \label{lemm:smooth.out.edge.singularities}
	Let $W \subset M$ be a precompact open set containing a nondegenerate $(n-2)$-skeleton $\cS \subset M$, and suppose that $g \in C^{2,\alpha}_{\loc}(W \setminus \cS)$, $\alpha \in [0,1]$, is an $\eta$-regular edge metric along $\reg \cS$ with data $(\eta, \beta, \sigma, \omega, \varrho, h)$ satisfying
	\begin{equation} \label{eq:smooth.out.edge.singularities.structural.i}
	    0 < \Lambda^{-1} \leq \inf_{\reg \cS} 2\pi (\beta+1) \leq \sup_{\reg \cS} 2\pi(\beta+1) \leq 2\pi,
	\end{equation}
	and
	\begin{multline} \label{eq:smooth.out.edge.singularities.structural.ii}
	    (\eta - 2 + \tfrac{4}{n})^{-1} + \Vert (\det \omega)^{-1} \Vert_{L^\infty} + \sum_{j=1}^2 \Vert \partial^j \beta \Vert_{L^\infty} + \sum_{j=0}^1 \Vert \partial^j \varrho \Vert_{L^\infty} \\
	    + \Vert \varrho^{-\eta} \partial^2 \varrho \Vert_{L^\infty} + \sum_{j=0}^2 + \Vert \partial^j \sigma \Vert_{L^\infty} + \Vert \partial^j \omega \Vert_{L^\infty} + \Vert \partial^j h \Vert_{L^\infty} \leq \Lambda.
	\end{multline}  
	See Definition \ref{defi:edge.singularity} for the notation. If $R(g) \geq 0$ on $W \setminus \cS$, then for every $W' \subset \subset W$ containing the $\varrho$-normal tubular neighborhood of $\reg \cS$ and every $\gamma > 0$,  there exist
	\begin{multline*}
	    \varepsilon_1 = \varepsilon_1(n, \Lambda, \gamma, \dist_g(W', \partial W)), \; c_1 = c_1(n, \Lambda, \dist_g(W', \partial W)), \\
	    \delta = \delta(n, \Lambda, \dist_g(W', \partial W)) > 0,
	\end{multline*}
	such that for every $\varepsilon \in (0, \varepsilon_1]$, there is a metric $\widehat{g}_\varepsilon$ on $W$ such that:
	\begin{enumerate}
	    \item $\widehat{g}_\varepsilon$ is $C^{2,\alpha}_{\loc}(W \setminus \sing \cS)$;
	    \item $\widehat{g}_\varepsilon = g$ on $W \setminus (W' \cap B_\varepsilon^g(\reg \cS))$;
	    \item $\|R(\widehat{g}_\varepsilon)_-\|_{L^{\frac{n}{2}+\delta}(W,g)} \leq \gamma$;
        \item $c_1^{-1} g \leq \widehat{g}_\varepsilon \leq c_1 g$ on $W$;
	    \item if $p \in \reg \cS$ is such that $\beta(p) < 0$ and $\mu > 0$ is such that
	        \[ |\beta(p)|^{-1} + \varrho(p)^{-1} \leq \mu, \]
	        then
	        \[ R(\widehat{g}_\varepsilon) \geq c_2 \varepsilon^{-2} \]
	        on $B^g_{c_2}(p) \cap B^g_{c_3 \varepsilon}(\reg \cS) \setminus B^g_{c_3 \varepsilon/2}(\reg \cS)$, with
            \begin{align*}
                c_2 & = c_2(n, \Lambda, \dist_g(W', \partial W), \mu) > 0, \\
                c_3 & = c_3(n, \Lambda, \dist_g(W', \partial W), \mu) > 0,
            \end{align*}
            and $B^g_{c_2}(p) \subset W'$.
    \end{enumerate}
\end{lemm}

\begin{figure}[htbp] \label{fig:cone.smoothing}
    \centering
    \includegraphics[scale=0.2]{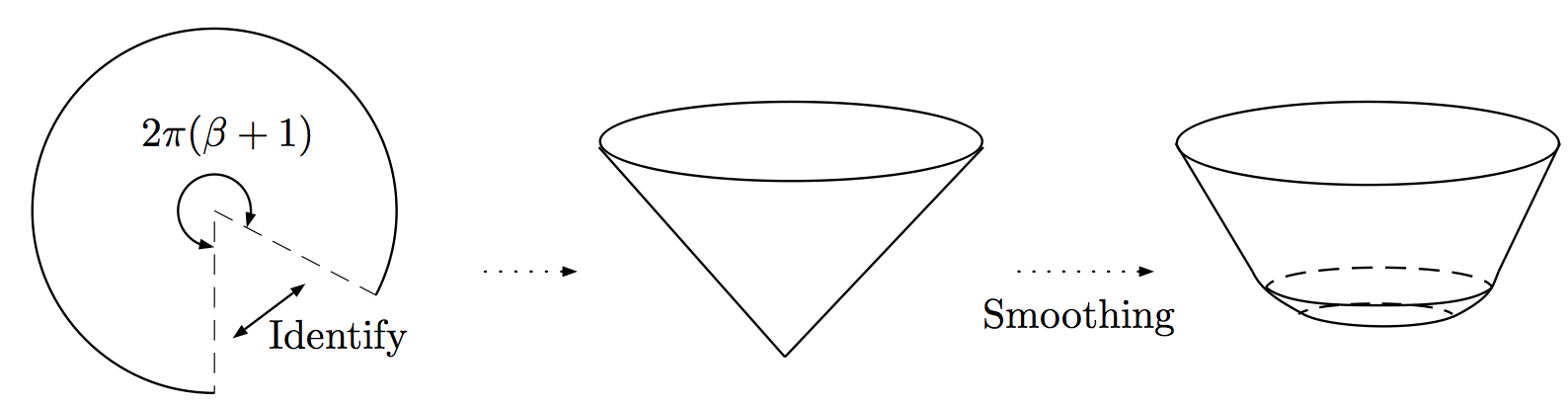}
    \caption{An illustration of a two-dimensional cone metric, and its smoothing procedure (Lemma \ref{lemm:smooth.out.edge.singularities}). Roughly speaking, we glue a flat disk onto the conical singularity, such that the metric is $L^\infty$, and the Gauss curvature is positive in the buffer region.}
\end{figure}

It is worth first looking at the special case in which $\cS = \overline{N}$ for some embedded $(n-2)$-dimensional submanifold $N$ (without boundary), and the edge singularity datum $h$ is identically zero, i.e.,
\begin{equation} \label{eq:smooth.out.edge.singularities.special.case}
    g = dr^2 + (\beta+1)^2 r^2 (d\theta + \sigma)^2 + \omega \text{ in } U \setminus N.
\end{equation}
As we'll see, all the interesting complications already arise in this situation. 

Let's temporarily divert our attention to metrics $\widetilde{g}$ of the computationally simpler form
\begin{equation} \label{eq:smooth.out.edge.singularities.simple.form}
    \widetilde{g} = f^2 dr^2 + r^2 (d\theta + \sigma)^2 + \widetilde{\omega} \text{ on } U \setminus N.
\end{equation}
Here, $f = f(r, y)$. We require the structural conditions
\begin{multline} \label{eq:smooth.out.edge.singularities.structural.iii}
    \Vert (\det \widetilde{\omega})^{-1} \Vert_{L^\infty} + \sum_{j=0}^2 \Vert \partial^j \sigma \Vert_{L^\infty} + \Vert \partial^j \widetilde{\omega} \Vert_{L^\infty}  \\
    + \Vert f^{-1} \Vert_{L^\infty(U)} + \Vert r \partial_r f \Vert_{L^\infty} + \Vert r \partial f \Vert_{L^\infty} + \Vert r^{2-\eta} \partial^2 f \Vert_{L^\infty} \leq \widetilde{\Lambda}.
\end{multline}
All partial derivatives except the one explicitly denoted $\partial_r$ are taken only with respect to $(y^1, \ldots, y^{n-2}) \in N$, but not in the two transversal polar directions. 

\begin{prop} \label{prop:metric.scalar.curvature}
    The scalar curvature $R(\widetilde{g})$ of metrics $\widetilde{g}$ of the form  \eqref{eq:smooth.out.edge.singularities.simple.form}, which are subject to the structural assumptions \eqref{eq:smooth.out.edge.singularities.structural.iii}, satisfies
    \begin{equation} \label{eq:smooth.out.edge.singularities.R.estimate}
        r^{2-\eta} | R(\widetilde{g}) - 2 r^{-1} f^{-3} \partial_r f | \leq c(n, \widetilde{\Lambda}),
    \end{equation}
    at all points $(r, \theta, y^1, \ldots, y^{n-2})$ with $r \leq R_0 = R_0(n, \widetilde{\Lambda})$. 
\end{prop}

\begin{proof}
    We circumvent a brute force computation by a slicing technique motivated by \eqref{eq:codim.1.scalar.curv.formula}. The family of hypersurfaces
    \[ N_r := \{ r = \text{const} \} \cap U, \; r > 0, \]
    forms a codimension-1 foliation of $U \setminus N$, which is orthogonal to the ambient vector field
    \[ \nu_r := f^{-1} \partial_r \]
    with respect to $\widetilde{g}$. In particular, the Gauss equation traced twice over $N_r$ gives:
    \begin{equation} \label{eq:metric.scalar.curvature.gauss}
        R(\widetilde{g}|_{N_r}) = R(\widetilde{g}) - 2 \langle \Ric(\widetilde{g}), \nu_r \otimes \nu_r \rangle + H_{N_r}^2 - |A_{N_r}|^2,
    \end{equation}
    where $H_{N_r}$, $A_{N_r}$ denote the mean curvature and second fundamental form of $N_r \subset (U, \widetilde{g})$. On the other hand, the Jacobi equation implies
    \begin{equation} \label{eq:metric.scalar.curvature.jacobi}
        \partial_r (H_{N_r}) = - \Delta_{\widetilde{g}|N_r} f - (\langle \Ric(\widetilde{g}), \nu_r \otimes \nu_r \rangle, + |A_{N_r}|^2) f.
    \end{equation}
    Together, \eqref{eq:metric.scalar.curvature.gauss}-\eqref{eq:metric.scalar.curvature.jacobi} yield:
    \begin{equation} \label{eq:metric.scalar.curvature}
        R(\widetilde{g}) = R(\widetilde{g}|_{N_r}) - 2 f^{-1} \partial_r (H_{N_r}) - 2 f^{-1} \Delta_{\widetilde{g}|N_r} f - H_{N_r}^2 - |A_{N_r}|^2.
    \end{equation}
    This is the quantity we wish to estimate, written out in terms of the slicing technique. Let's fix $r > 0$ small and estimate the right hand side of \eqref{eq:metric.scalar.curvature}.
    
    Recall that
    \begin{equation} \label{eq:metric.scalar.curvature.induced}
        \widetilde{g}|_{N_r} = r^2 (d\theta + \sigma)^2 + \omega,
    \end{equation}
    and that, by the definition of second fundamental forms (here, $\mathcal{L}$ is the Lie derivative on 2-tensors),
    \begin{equation} \label{eq:metric.scalar.curvature.sff}
        A_{N_r} = \tfrac{1}{2} \mathcal{L}_{\nu_r} \widetilde{g} = \tfrac{1}{2} f^{-1} \mathcal{L}_{\partial_r}  (\widetilde{g}|_{N_r}) = \tfrac{1}{rf} r^2 (d\theta + \sigma)^2.
    \end{equation}
    It will be convenient to pick out vector fields
    \begin{multline*}
        \mathbf{v}_1, \ldots, \mathbf{v}_{n-2} \in \Gamma(TN^{n-2}) \text{ to be an } \\
        \omega \text{-orthonormal frame on } N^{n-2}.
    \end{multline*}
    We emphasize, that these are orthonormal on $N^{n-2}$ with a metric \emph{other} than the model metric $\omega \in \operatorname{Met}(N^{n-2})$. This modified metric was chosen specifically because, now,
    \begin{multline*}
        r^{-1} \partial_\theta, \mathbf{v}_1 - \sigma(\mathbf{v}_1) \partial_\theta, \ldots, \mathbf{v}_{n-2} - \sigma(\mathbf{v}_{n-2}) \partial_\theta \\
        \text{are a } \widetilde{g}_{N_r}\text{-orthonormal frame on } N_r.
    \end{multline*}
    By repeated use of \eqref{eq:metric.scalar.curvature.sff}, we find:
    \begin{equation} \label{eq:metric.scalar.curvature.sff.thetatheta}
        A_{N_r}(r^{-1} \partial_\theta, r^{-1} \partial_\theta) = \tfrac{1}{rf};
    \end{equation}
    \begin{align} \label{eq:metric.scalar.curvature.sff.vv}
        & A_{N_r}(\mathbf{v}_\ell - \sigma(\mathbf{v}_\ell) \partial_\theta, \mathbf{v}_m - \sigma(\mathbf{v}_m) \partial_\theta) \nonumber \\
        & \qquad = A_{N_r}(\mathbf{v}_\ell, \mathbf{v}_m) - \sigma(\mathbf{v}_\ell) A_{N_r}(\partial_\theta, \mathbf{v}_m) \nonumber \\
        & \qquad \qquad - \sigma(\mathbf{v}_m) A(\mathbf{v}_\ell, \partial_\theta) + \sigma(\mathbf{v}_\ell) \sigma(\mathbf{v}_m) A_{N_r}(\partial_\theta, \partial_\theta) \nonumber \\
        & \qquad = \tfrac{r}{f} \sigma(\mathbf{v}_\ell) \sigma(\mathbf{v}_m) - 2 \tfrac{r}{f} \sigma(\mathbf{v}_\ell) \sigma(\mathbf{v}_m) + \tfrac{r}{f} \sigma(\mathbf{v}_\ell) \sigma(\mathbf{v}_m) \nonumber \\
        & \qquad = 0 \text{ for } \ell, m \in \{ 1, \ldots, n-2 \};
    \end{align}
    \begin{align} \label{eq:metric.scalar.curvature.sff.thetav}
        & A_{N_r}(r^{-1} \partial_\theta, \mathbf{v}_\ell - \sigma(\mathbf{v}_\ell) \partial_\theta) \nonumber \\
        & \qquad = A_{N_r}(r^{-1} \partial_\theta,  \mathbf{v}_\ell) - A_{N_r}(r^{-1} \partial_\theta, \sigma(\mathbf{v}_\ell) \partial_\theta) \nonumber \\
        & \qquad = \tfrac{r}{f} \sigma(\mathbf{v}_\ell) - \tfrac{r}{f} \sigma(\mathbf{v}_\ell) \nonumber \\
        & \qquad = 0 \text{ for } \ell \in \{1, \ldots, n-2\}.
    \end{align}
    Altogether, \eqref{eq:metric.scalar.curvature.sff.thetatheta}-\eqref{eq:metric.scalar.curvature.sff.thetav} imply
    \begin{equation*} 
        |A_{N_r}| = H_{N_r} = \tfrac{1}{rf},
    \end{equation*}
    and thus
    \begin{equation} \label{eq:metric.scalar.curvature.i}
        2f^{-1} \partial_r (H_{N_r}) + H_{N_r}^2 + |A_{N_r}|^2 = - 2 r^{-1} f^{-3} \partial_r f.
    \end{equation}
    In particular, three out of five terms in \eqref{eq:metric.scalar.curvature} cancel out. 
    
    Next, we seek to understand $R(\widetilde{g}|_{N_r})$, which denotes the scalar curvature of the $(n-1)$-dimensional manifold $(N_r, \widetilde{g}|_{N_r})$, with $\widetilde{g}|_{N_r}$ given explicitly in \eqref{eq:metric.scalar.curvature.induced}. We re-employ the slicing technique; this time we use the fact that
    \[ N_{r,\theta} := \{ \theta = \text{const} \} \cap N_r \]
    is a codimension-1 foliation of $N_r$, whose induced metrics are given by
    \begin{equation} \label{eq:metric.scalar.curvature.induced.2}
        \widetilde{g}|_{N_{r,\theta}} = \omega + (r\sigma)^2.
    \end{equation}
    If $\nu_{r,\theta} \in \Gamma(TN_r)$ denotes the unit normal vector field to the foliation, then, arguing as before, we have
    \begin{multline} \label{eq:metric.scalar.curvature.theta}
        R(\widetilde{g}|{N_r}) = R(\widetilde{g}|{N_{r,\theta}}) - 2 \langle \nu_{r,\theta}, \partial_\theta \rangle_{\widetilde{g}}^{-1} \mathcal{L}_{\nu_{r,\theta}} (H_{N_{r,\theta}}) \\
        - 2 \langle \nu_{r,\theta}, \partial_\theta \rangle_{\widetilde{g}}^{-1} \Delta_{\widetilde{g}|N_{r,\theta}} \langle \nu_{r,\theta}, \partial_\theta \rangle_{\widetilde{g}} - H_{N_{r,\theta}}^2 - |A_{N_{r,\theta}}|^2.
    \end{multline}
    Note that, unlike the previous slicing application, $\partial_\theta$ is no longer orthogonal to the foliation. Instead, the unit normal vector field $\nu_{r,\theta}$ is proportional to
    \[ \partial_\theta + \sum_{\ell=1}^{n-2} \alpha_\ell \mathbf{v}_\ell \]
    for some coefficients $\alpha_1, \ldots, \alpha_{n-2} : N^{n-2} \to \RR$; the vector fields $\mathbf{v}_\ell$ are the same as before. The coefficients $\alpha_1, \ldots, \alpha_{n-2}$ are such that
    \[ \langle \nu_{r,\theta}, \mathbf{v}_1 \rangle_{\widetilde{g}} = \ldots = \langle \nu_{r,\theta}, \mathbf{v}_{n-2} \rangle_{\widetilde{g}} = 0. \] This is a uniformly invertible $(n-2) \times (n-2)$ linear system for small enough $r \leq r_0 = r_0(n, \widetilde{\Lambda})$. Recalling \eqref{eq:smooth.out.edge.singularities.structural.iii}, the linear system readily implies:
    \begin{equation} \label{eq:metric.scalar.curvature.theta.i}
        \sum_{\ell=1}^{n-2} |\alpha_\ell| + r |\partial \alpha_\ell| + r^2 |\partial^2 \alpha_\ell| \leq c(n, \widetilde{\Lambda}) r^2.
    \end{equation}
    In particular, the unit normal vector field is
    \[ \nu_{r,\theta} = (1 + \zeta) \left( \partial_\theta + \sum_{\ell=1}^{n-2} \alpha_\ell \mathbf{v}_\ell \right), \]
    with 
    \begin{equation} \label{eq:metric.scalar.curvature.theta.ii}
        |\zeta| + r |\partial \zeta| + r^2 |\partial^2 \zeta| \leq c(n, \widetilde{\Lambda}) r^2.
    \end{equation}
    Combined,   \eqref{eq:metric.scalar.curvature.induced.2}, \eqref{eq:metric.scalar.curvature.theta.i}, and \eqref{eq:metric.scalar.curvature.theta.ii}, imply a uniform bound on the right hand side of \eqref{eq:metric.scalar.curvature.theta}. Thus,
    \begin{equation} \label{eq:metric.scalar.curvature.ii}
        |R(\widetilde{g}|N_{r,\theta})| \leq c(n, \widetilde{\Lambda}).
    \end{equation}
    
    Finally, the last remaining term of \eqref{eq:metric.scalar.curvature}, $\Delta_{\widetilde{g}|N_r} f$, can be estimated directly by    \eqref{eq:smooth.out.edge.singularities.structural.iii}:
    \begin{equation} \label{eq:metric.scalar.curvature.iii}
        r^{2-\eta} |\Delta_{\widetilde{g}|N_r} f| \leq c(n, \widetilde{\Lambda}).
    \end{equation} 
    The proposition follows by plugging \eqref{eq:metric.scalar.curvature.i}, \eqref{eq:metric.scalar.curvature.ii}, \eqref{eq:metric.scalar.curvature.iii} into \eqref{eq:metric.scalar.curvature}.
\end{proof}

\begin{proof}[Proof of Lemma \ref{lemm:smooth.out.edge.singularities}]

Let us first see how Proposition \ref{prop:metric.scalar.curvature} fits into our simplified smoothing lemma situation, i.e., $\cS = \overline{N}$ and $h \equiv 0$. Let's fix a smooth cutoff function $\zeta : [0, 1] \to [0, 1]$ such that
\[ \zeta \equiv 0 \text{ on } [0,\tfrac{1}{3}], \; \zeta \equiv 1 \text{ on } [\tfrac{2}{3}, 1], \; 0 \leq \zeta' \leq 6, \; \zeta' = 1 \text{ on } [\tfrac{4}{9}, \tfrac{5}{9}]. \]
Define $f_\ep(r, y)$, $\ep > 0$:
\begin{equation} \label{eq:smooth.out.edge.singularities.f}
    f_\ep(r, y) := 1 + \zeta(\ep^{-1} \varrho(y)^{-1} r)\Big[ (1 + \beta(y))^{-1} - 1 \Big].
\end{equation}
From \eqref{eq:smooth.out.edge.singularities.structural.i}, \eqref{eq:smooth.out.edge.singularities.structural.ii}, and the defining properties of $\zeta$:
\begin{equation} \label{eq:smooth.out.edge.singularities.f.i}
    f_\ep \geq 1, \; 0 \leq r \partial_r f_\ep \leq 6, 
\end{equation}
\begin{equation} \label{eq:smooth.out.edge.singularities.f.ii}
    f_\ep = 1 \text{ for } r \leq \tfrac{1}{3} \ep \varrho, \; f_\ep = (\beta+1)^{-1} \text{ for } r \geq \tfrac{2}{3} \ep \varrho,
\end{equation}
\begin{equation} \label{eq:smooth.out.edge.singularities.f.iii}
    r \partial_r f_\ep(r, y) = (1+\beta)^{-1} - 1 \text{ for } \ep \tfrac{4}{9} \varrho(y) \leq r \leq \tfrac{5}{9} \ep \varrho(y),
\end{equation}
\begin{equation} \label{eq:smooth.out.edge.singularities.f.iv}
    |r \partial f_\ep| + |r^{2-\eta} \partial^2 f_\ep| \leq c(\Lambda).
\end{equation}
Setting
\begin{equation} \label{eq:smooth.out.edge.singularities.eps.regularization}
    \widetilde{g}_\ep := f_\ep^2 dr^2 + r^2 (d\theta + \sigma)^2 + (\beta+1)^{-2} \omega,
\end{equation}
it follows from \eqref{eq:smooth.out.edge.singularities.f.i}-\eqref{eq:smooth.out.edge.singularities.f.iv} and  \eqref{eq:smooth.out.edge.singularities.structural.i}-\eqref{eq:smooth.out.edge.singularities.structural.ii} that $\widetilde{g}_\ep$ is of the form \eqref{eq:smooth.out.edge.singularities.simple.form} and satisfies the structural assumptions \eqref{eq:smooth.out.edge.singularities.structural.iii}.

We'll verify that, for sufficiently small $\ep > 0$, the conformal metric
\[ \widehat{g}_\ep := (\beta+1)^2 \widetilde{g}_\ep \]
is the metric postulated by Lemma \ref{lemm:smooth.out.edge.singularities}. Without loss of generality,
\[ \dist_g(W', \partial W) \geq 1, \; \varrho \leq 1 \text{ on } N. \]
Conclusions (1), (2), (4) of Lemma \ref{lemm:smooth.out.edge.singularities} is an immediate consequence of \eqref{eq:smooth.out.edge.singularities.structural.ii} and the definitions of $\widetilde{g}_\ep$, $\widehat{g}_\ep$. Now we prove conclusion (5). If $p$ is as in the statement of the Lemma, then by Proposition \ref{prop:metric.scalar.curvature} and \eqref{eq:smooth.out.edge.singularities.f.ii},
\begin{equation} \label{eq:smooth.out.edge.singularities.scalar.curvature.i}
    r^{2-\eta} |R(\widetilde{g}_\ep) - 2r^{-2} f_\ep^{-3} ((1+\beta(p))^{-1} - 1)| \leq c,
\end{equation}
whenever $r \in [\tfrac{4}{9} \ep \varrho(p), \tfrac{5}{9} \ep \varrho(p)]$. This readily implies conclusion (5). Finally, we move on to conclusion (3). By Proposition \ref{prop:metric.scalar.curvature}, we have
\[ R(\widetilde{g}_\ep)_- \leq c r^{-2+\eta}, \]
so the conformal metric $\widehat{g}_\ep = (\beta+1)^2 \widetilde{g}_\ep$ satisfies
\[ R(\widehat{g}_\ep) = (\beta+1)^{\tfrac{n+2}{2}} \Big[ \frac{4(1-n)}{n-2} \Delta_{\widetilde{g}_\ep} + R(\widetilde{g}_\ep) \Big] (\beta+1)^{\tfrac{n-2}{2}}. \]
Since $\beta$ has no dependence on $r$, $\theta$, and is uniformly $C^2$ in $(y^1, \ldots, y^{n-2})$:
\[ R(\widehat{g}_\ep)_- \leq c(1 + R_-(\widetilde{g}_\ep)) \leq c(1 + r^{-2+\eta}) \leq c r^{-2+\eta}, \]
where the last inequality follows from our assumption that $\varrho \leq 1$. In particular, if we denote the $\ep \varrho$-tubular neighborhood of $N$ by $U_\ep$, we have, from the coarea formula, that
\begin{align*}
    \Vert R(\widehat{g}_\ep)_- \Vert_{L^q(W,g)}^q & = \Vert R(\widehat{g}_\ep)_- \Vert_{L^q(U_\ep,g)}^q \\
    & \leq c \int_{U_\ep} (r^{-2+\eta})^q d\Vol_{g} \\
    & \leq c \int_N \int_0^{\ep \varrho(y)} r^{q(-2+\eta)+1} \, dr \, d\mu_\omega(y) \\
    & = c \int_N \left[ \tfrac{r^{q(-2+\eta)+2}}{q(-2+\eta)+2} \right]_{r=0}^{\ep \varrho(y)} \, d\mu_\omega(y) \\
    & \leq c (\ep \Vert \varrho \Vert_{C^0(N)})^{q(-2+\eta) + 2},
\end{align*}
provided
\[ q(-2+\eta) + 2 > 0 \iff q < \tfrac{2}{2-\eta}. \]
In the chain of inequalities above, $c$ denotes a constant depending on $n$ and $\Lambda$, which varies from line to line. Since $\eta \geq \Lambda^{-1} + 2 - \tfrac{4}{n}$, it follows that
\[ q < \tfrac{2}{\tfrac{4}{n} - \Lambda^{-1}}, \]
and conclusion (3) follows. This completes the proof of the lemma in the special case when $\cS = \overline{N}$ and $h \equiv 0$.

Let's generalize to allow $h \not \equiv 0$ in
\[ g = dr^2 + (\beta+1)^2 r^2 (d\theta+\sigma)^2 + \omega + r^{1+\eta} h. \]
We will regularize in two steps, leading up to
\[ \widehat{g}_\ep := (\beta+1)^2 \widetilde{g}_\ep + (\beta+1)^2 f_\ep^2 r^{1+\eta} h, \]
where $f_\ep$ is as in \eqref{eq:smooth.out.edge.singularities.f} and $\widetilde{g}_\ep$ as in \eqref{eq:smooth.out.edge.singularities.eps.regularization}. The first step, studying $(\beta+1)^2 \widetilde{g}_\ep$, is the step we carried out above. Now, a crude estimate that relies on \eqref{eq:smooth.out.edge.singularities.structural.iii} shows that when $\xi$ is a $C^2_{\operatorname{loc}}(U \setminus N)$ 2-tensor, which in Euclidean coordinates (recall Definition \ref{defi:edge.singularity}) is controlled by
\[ |\xi_{\alpha \beta}| + r |\partial_\alpha \xi_{\beta \gamma}| + r^2 |\partial_\alpha \partial_\beta \xi_{\gamma \delta}| \leq \epsilon \]
and $\epsilon > 0$ sufficiently small, then
\begin{equation} \label{eq:metric.scalar.curvature.crude.estimate}
    r^2 |R\big( (\beta+1)^2 \widetilde{g}_\ep + \xi \big) - R\big( (\beta+1)^2 \widetilde{g}_\ep \big)| \leq c(n, \Lambda) \epsilon.
\end{equation}
But note that
\[ \xi := \widehat{g}_\ep - (\beta+1)^2 \widetilde{g}_\ep = (\beta+1)^2 f_\ep^2 r^{1+\eta} h \] satisfies
\[ |\xi_{\alpha \beta}| + r |\partial_\alpha \xi_{\beta \gamma}| + r^2 |\partial_\alpha \partial_\beta \xi_{\gamma \delta}| \leq c(n, \Lambda) r^\eta, \]
and $\eta > 0$, which applied to \eqref{eq:metric.scalar.curvature.crude.estimate} tells us that $R(\widehat{g}_\ep)$ has precisely the same behavior now as in \eqref{eq:smooth.out.edge.singularities.scalar.curvature.i}, so the result follows as before.


Finally, we deal with the most general case, where $g$ can be of general edge type, and the skeleton $\cS$ consists of more than just one piece; i.e., $\cS = \overline{N}_1 \cup \ldots \cup \overline{N}_k$. Since we're assuming $\cS$ is nondegenerate, it follows that the pieces $N_1, \ldots, N_k$ can be separated from each other with $\varrho$-tubular neighborhoods that decay with
\begin{equation} \label{eq:smooth.out.edge.singularities.component.dgeneration}
    \varrho \sim \dist_g(\cdot, \partial N_1 \cup \ldots \cup \partial N_k).
\end{equation}
In particular, we may apply the lemma to each component $N_1, \ldots, N_k$ individually with a modified $\Lambda$ that also accounts for the linear decay \eqref{eq:smooth.out.edge.singularities.component.dgeneration}, and then glue all the metrics together since they agree away from their degenerating tubular neighborhoods by virtue of the rightmost equality in \eqref{eq:smooth.out.edge.singularities.f.ii}.
\end{proof}

\section{Almost positive scalar curvature} \label{sec:lq.positive.scalar.curvature}

The following lemma will play a key and recurring role in this work, stating that $C^{2,\alpha}_{\loc} \cap L^\infty$ metrics with little negative scalar curvature and sufficiently much positive scalar curvature are conformally equivalent to metrics with positive scalar curvature of the same regularity.

\begin{lemm} \label{lemm:lq.positive.scalar.curvature}
	Suppose $M^n$ is closed, $g_0$ is a smooth background metric on $M$, $g$ is an $L^\infty(M) \cap C^{2,\alpha}_{\loc}(M \setminus \cS)$, $\alpha \in (0, 1)$, $\cS \subset M$ is compact, $\Vol_g(\cS) = 0$, and $\Lambda^{-1} g_0 \leq g \leq \Lambda g_0$. If $\chi \in C^\alpha_{\loc}(M \setminus \cS) \cap L^q(M, g)$ with $q > \frac{n}{2}$,
	\[ \chi \leq R(g), \; \Vert \chi_- \Vert_{L^{n/2}(M,g)} \leq \delta_0, \]
	then there exists $u \in C^{2,\alpha}_{\loc}(M \setminus \cS) \cap C^0(M)$, $u > 0$, such that
	\begin{multline} \label{eq:lq.positive.scalar.curvature}
		\inf_{M \setminus \cS} u^{\frac{4}{n-2}} R(u^{\frac{4}{n-2}} g) \geq \frac{1}{c_0^2 \Vol_g(M)} \left(\int_M \chi_+ \, d\Vol_g - c_0^4 \int_M \chi_{-}\,d\Vol_g\right)\\  \text{ and } \sup_M u \leq c_0 \inf_M u,
	\end{multline}
	where $\delta_0 = \delta_0(g_0, \Lambda) > 0$, $c_0 = c_0(g_0, \Lambda, q, \Vert \chi \Vert_{L^q(M, g, \Lambda)}) \geq 1$. 
\end{lemm}
\begin{proof}
	We construct, using the direct method, the principal eigenvalue of the operator $- \frac{4(n-1)}{n-2} \Delta_g + \chi$ on $S$. Namely, we minimize
	\begin{equation} \label{eq:lq.positive.scalar.curvature.i}
		\Vert f \Vert_{L^2(M,g)}^{-2} \int_M \frac{4(n-1)}{n-2} \Vert \nabla^g f \Vert_g^2 + \chi |f|^2 \, d\Vol_g,
	\end{equation}
	over $f \in L^2(M, g)$, $f \not \equiv 0$. From the Poincar\'e-Sobolev inequality,
    \begin{align*}
        & \left( \int_M f^{\frac{2n}{n-2}} \, d\Vol_g \right)^{\frac{n-2}{n}} \leq C_1 \int_M \Vert  \nabla^g f \Vert_g^2 + |f|^2 \, d\Vol_g \\
        & \qquad \implies \int_M \Vert \nabla^g f \Vert_g^2 \, d\Vol_g \geq C_1^{-1} \left( \int_M |f|^{\frac{2n}{n-2}} \, d\Vol_g \right)^{\frac{n-2}{n}} - \int_M |f|^2 \, d\Vol_g
    \end{align*}
    for $C_1 = C_1(g_0, \Lambda) > 0$. From H\"older's inequality,
    \[ \int_M \chi |f|^2 \, d\Vol_g \geq - \delta_0 \left( \int_M |f|^{\frac{2n}{n-2}} \, d\Vol_g \right)^{\tfrac{n-2}{n}} \]
    and the lower bound on \eqref{eq:lq.positive.scalar.curvature.i} follows as long as we require $\delta_0$ to be small enough depending on $g_0$, $\Lambda$. From functional analysis, minimizing \eqref{eq:lq.positive.scalar.curvature.i} yields some $u \in W^{1,2}(M, g)$, $u \geq 0$ $g$-a.e. on $M$, that satisfies, for some $\lambda \in \RR$,
    \begin{equation} \label{eq:lq.positive.scalar.curvature.ii}
    	- \frac{4(n-1)}{n-2} \Delta_g u + \chi u = \lambda u \text{ on } M,
    \end{equation}
    in the weak sense. From elementary elliptic PDE theory, $u \in C^{2,\alpha}_{\loc}(M \setminus \cS)$. From De Giorgi-Nash-Moser theory,
    \[ u \in C^{0,\theta}(M), \text{ and } \lambda \geq -\Lambda, \; \Lambda = \Lambda(g_0, \Lambda, q, \Vert \chi \Vert_{L^q(M,g)}) > 0. \]
    (The precise $\theta \in (0, 1)$ isn't relevant.) The inequality
    \begin{equation} \label{eq:lq.positive.scalar.curvature.iii}
    	 \sup_M u \leq c_0 \inf_M u
    \end{equation}
    with $c_0 = c_0(g_0, \Lambda, q, \Vert \chi \Vert_{L^q(M,g)})$ follows from Moser's Harnack inequality. From the variational characterization of \eqref{eq:lq.positive.scalar.curvature.ii} as a minimizer of \eqref{eq:lq.positive.scalar.curvature.i}, and from \eqref{eq:lq.positive.scalar.curvature.iii}, we see that
    \begin{align}
    	\lambda 
    	    & = \Vert u \Vert_{L^2(M,g)}^{-2} \int_M \frac{4(n-1)}{n-2} \Vert \nabla^g u \Vert_g^2 + \chi |u|^2 \, d\Vol_g \nonumber \\
    	    & \geq \Vert u \Vert^{-2}_{L^2(M,g)} \int_M \chi_+ |u|^2 \, d\Vol_g - \Vert u \Vert^{-2}_{L^2(M,g)} \int_M \chi_- |u|^2 \, d\Vol_g \nonumber \\
    	    & \geq \inf_M u^2 \cdot \Vert u \Vert_{L^2(M,g)}^{-2} \int_M \chi_+ \, d\Vol_g - \sup_M u^2 \cdot \Vert u \Vert^{-2}_{L^2(M,g)} \int_M \chi_- \, d\Vol_g \nonumber \\
    		& \geq c_0^{-2} \Vol_g(M,g)^{-1} \left(\int_M \chi_+ \, d\Vol_g- c_0^4 \int_M \chi_{-}\,d\Vol_g\right). \label{eq:lq.positive.scalar.curvature.iv}
    \end{align}
    Thus, from the scalar curvature conformal transformation formula and \eqref{eq:lq.positive.scalar.curvature.ii},
    \begin{align*}
    	R(u^{\frac{4}{n-2}} g) & = u^{-\frac{n+2}{n-2}} \left( - \frac{4(n-1)}{n-2} \Delta_g u + R(g) u \right) \\
    		& = u^{-\frac{4}{n-2}} (R(g)-\chi+\lambda) \geq \lambda u^{-\frac{4}{n-2}} \text{ on } M \setminus \cS,
    \end{align*}
    and the result follows from \eqref{eq:lq.positive.scalar.curvature.iv}.
\end{proof}

We obtain, as a direct corollary, the following rigidity result that extends a well-known (to the experts of the field) result from the smooth case to a general singular setting: nonnegative scalar curvature can be conformally transformed into positive scalar curvature, as long as the original metric isn't scalar-flat.

\begin{coro} \label{coro:positive.scalar.curvature}
	Suppose $M^n$ is closed, $g$ is an $L^\infty(M) \cap C^{2,\alpha}_{\loc}(M \setminus \cS)$ metric, $\alpha \in (0,1)$, $\cS \subset M$ is compact, and $\Vol_g(\cS) = 0$. If $R(g) \geq 0$ on $M \setminus \cS$, and $R(g) \not \equiv 0$, then
	\[ R(u^{\frac{4}{n-2}} g) > 0 \text{ on } M \setminus \cS \]
	for some $u \in C^{2,\alpha}_{\loc}(M \setminus \cS) \cap C^0(M)$, $u > 0$.
\end{coro}

\begin{rema}
    We will later show that for particular kinds of singular behavior, we can construct \emph{everywhere} smooth metrics with positive scalar curvature, at the expense of leaving the conformal class of $g$. This is essentially the content of Theorems \ref{theo:main.3D}, \ref{theo:main.highD}, and Corollary \ref{coro:main.3D}.
\end{rema}

\section{Smoothing edge singularities, II} \label{sec:smooth.out.n2.faces}

\begin{prop} \label{prop:smooth.out.n2.faces}
	Suppose $M^n$ is closed, $\sigma(M) \leq 0$, $\cS \subset M$ is a nondegenerate $(n-2)$-skeleton, and $g \in L^\infty(M) \cap C^{2,\alpha}_{\loc}(M \setminus \cS)$, $\alpha \in [0,1]$. Assume $g$ is an $\eta$-regular edge metric along $\reg \cS$ with $\eta > 2 - \tfrac{4}{n}$ and cone angles
	\[ 0 < \inf_{\reg \cS} 2\pi(\beta+1) \leq \sup_{\reg \cS} 2\pi(\beta+1) \leq 2\pi. \]
	If $R(g) > 0$ on $M \setminus \cS$ and either
	\begin{enumerate}
		\item $R(g) \not \equiv 0$ on $M \setminus \cS$, or
		\item $2\pi (\beta+1) \not \equiv 2\pi$ on $\reg \cS$, 
	\end{enumerate}
	then there exists an $L^\infty(M) \cap C^{2,\alpha}_{\loc}(M \setminus \sing \cS)$ metric $\widetilde{g}$ with
	\[ R(\widetilde{g}) > 0 \text{ on } M \setminus \sing \cS. \]
\end{prop}

We need to introduce some more notation. For $s > 0$, define 
\begin{equation} \label{eq:phi.values}
	\phi(\cdot; s) :\RR \to \RR \text{ with } \phi(x; s) = \begin{cases} x \text{ for } x \in (-\infty, s], \\ 2s \text{ for } x \in [3s, \infty), \end{cases}
\end{equation}
with
\begin{equation} \label{eq:phi.conditions}
	\frac{\partial}{\partial x} \phi(x; s) \geq 0 \text{ and } \phi(x; s) \leq x \text{ for all } x \in \RR, s > 0,
\end{equation}
and, for $q \in (\frac{n}{2}, n)$ fixed for the remainder of the paper, and $\varepsilon > 0$,
\begin{equation} \label{eq:zeta.values}
	\zeta(\cdot; \ep) : M \to \RR \text{ with } \zeta(x; \ep) = \begin{cases} 1 \text{ for } x \not \in B^g_{2\ep}(\cS), \\ \ep^{-2/q} \text{ for } x \in B^g_{\ep}(\cS), \end{cases}
\end{equation}
such that
\begin{equation} \label{eq:zeta.conditions}
	|\zeta(x; \ep)| \leq \ep^{-2/q} \text{ for all } x \in M, \ep > 0. 
\end{equation}

\begin{proof}[Proof of Proposition {\ref{prop:smooth.out.n2.faces}}]
	Let $\widehat{g}_\varepsilon$ be as in Lemma \ref{lemm:smooth.out.edge.singularities} above, for a choice of $\gamma > 0$ that is yet to be determined.
	
	\begin{claim}
    	We have
    	\[ \limsup_{\varepsilon \to 0} \Vert \phi(R(\widehat{g}_\epsilon); \zeta(\cdot; \ep)) \Vert_{L^q(M,g)} < \infty. \]
    \end{claim}
    
    \begin{rema}
        The $L^q$ norm can be taken with respect to the measure induced by any one of $g$, $\widehat{g}_\varepsilon$, since they are uniformly equivalent by Lemma \ref{lemm:smooth.out.edge.singularities} (4).
    \end{rema}
    
    \begin{proof}[Proof of Claim]
		We estimate the integral by splitting up $M$ into the region of negative scalar curvature (which is controlled by Lemma \ref{lemm:smooth.out.edge.singularities}), the tubular neighborhood $B^g_{2\ep}(\cS)$ (which is controlled by the codimension of $\cS$), and the remainder:
		\begin{align*}
			& \Vert \phi(R(\widehat{g}_\varepsilon); \zeta(\cdot; \varepsilon)) \Vert_{L^q(M,g)}^q \\
			& \qquad \leq \Vert R(\widehat{g}_\varepsilon)_- \Vert_{L^q(M,g)}^q + \int_{B^g_{2\varepsilon}(\cS)} |2 \varepsilon^{-2/q}|^q \, d\Vol_g \\
			& \qquad \qquad + \int_{M \setminus B^g_{2\varepsilon}(\cS)} | \min\{R(\widehat{g}_\varepsilon)_+, 2 \}|^q \, d\Vol_g \\
			& \qquad \leq \gamma^q + 2^q \varepsilon^{-2} \Vol_g(B^g_{2\varepsilon}(\cS)) + 2^q \Vol_g(M),
		\end{align*}
		which is uniformly bounded as $\varepsilon \to 0$ when $\cS$ is an $(n-2)$-skeleton, i.e., one with codimension $\geq 2$.
	\end{proof}
	
	We now apply Lemma  \ref{lemm:lq.positive.scalar.curvature} with $\widehat{g}_\varepsilon$ in place of $g$, $\sing \cS$ in place of $\cS$, and $\chi = \phi(R(\widehat{g}_\varepsilon); \zeta(\cdot; \varepsilon))$. Note that the constant $\delta_0$ in Lemma \ref{lemm:lq.positive.scalar.curvature} is independent of $\varepsilon \to 0$, since the metrics $\widehat{g}_\ep$, $g$ are all uniformly equivalent. It remains to show
	\begin{equation} \label{eq:smooth.out.n2.faces.i}
		\int_M \phi(R(\widehat{g}_\varepsilon); \zeta(\cdot; \varepsilon))_+ \, d\Vol_{\widehat{g}_\varepsilon} - c_0^4 \int_M \phi(R(\widehat{g}_\varepsilon); \zeta(\cdot; \varepsilon))_{-} \, d\Vol_{\widehat{g}_\varepsilon}> 0
	\end{equation}
	for sufficiently small $\varepsilon > 0$.
	
	Separating the positive scalar curvature regions from the negative ones, recalling $R(\widehat{g}_\varepsilon) \geq 0$ on $M \setminus B^g_\varepsilon(\reg \cS)$ by Lemma \ref{lemm:smooth.out.edge.singularities} (2), and $\Vert R_-(\widehat{g}_\varepsilon) \Vert_{L^q} \leq \gamma$ by Lemma \ref{lemm:smooth.out.edge.singularities} (3),
    \begin{align}
    	& \int_M \phi(R(\widehat{g}_\varepsilon); \zeta(\cdot; \varepsilon))_+ \, d\Vol_{\widehat{g}_\varepsilon} - c_0^4 \int_M \phi(R(\widehat{g}_\varepsilon); \zeta(\cdot; \varepsilon))_{-} \, d\Vol_{\widehat{g}_\varepsilon} \nonumber \\
        & \qquad \geq \int_M \phi(R_+(\widehat{g}_\varepsilon); \zeta(\cdot; \varepsilon^{-1})) \, d\Vol_{\widehat{g}_\varepsilon} - c_0^4 \Vert R_-(\widehat{g}_\varepsilon) \Vert_{L^1(M,\widehat{g}_\varepsilon)} \nonumber \\
        & \qquad \geq \int_M \phi(R_+(\widehat{g}_\varepsilon); \zeta(\cdot; \varepsilon^{-1})) \, d\Vol_{\widehat{g}_\varepsilon} - c_0^4\gamma \Vol_{\widehat{g}_\varepsilon} (B^g_\varepsilon(\cS))^{(q-1)/q} \nonumber \\
        & \qquad \geq \int_M \phi(R_+(\widehat{g}_\varepsilon); \zeta(\cdot; \varepsilon^{-1})) \, d\Vol_{\widehat{g}_\varepsilon} - \gamma C_1 \varepsilon^{2(q-1)/q}, \label{eq:smooth.out.n2.faces.ii}
    \end{align}
    where $C_1 = C_1(\cS, g_0, \Lambda) > 0$, and $g_0$ is some fixed background smooth metric on $M$. Note that
    \begin{align*}
        & \liminf _{\varepsilon \to 0} \int_M \phi(R_+(\widehat{g}_\varepsilon); \zeta(\cdot; \varepsilon)) \, d\Vol_{\widehat{g}_\varepsilon} \\
        & \qquad \geq \liminf_{\varepsilon \to 0} \int_{M \setminus B^g_{2\varepsilon}(\reg \cS)} \phi(R_+(g); 1) \, d\Vol_{\widehat{g}_\varepsilon} = \int_M \phi(R(g); 1) \, d\Vol_g.
    \end{align*}
    In particular, if $R(g) \not \equiv 0$ on $M \setminus \cS$, then \eqref{eq:smooth.out.n2.faces.ii} implies \eqref{eq:smooth.out.n2.faces.i}, and we're done.
    
    Alternatively, when $R(g) \equiv 0$ on $M \setminus \cS$, suppose that $\beta \not \equiv 0$ on $\reg \cS$. By Lemma \ref{lemm:smooth.out.edge.singularities} (5), there exists $p$ with
    \[ \phi(R_+(\widehat{g}_\varepsilon); \zeta(\cdot; \varepsilon)) = 2 \varepsilon^{-2/q} \text{ on } B^g_{c_3}(p) \cap B^g_{c_4 \varepsilon}(\reg \cS) \setminus B^g_{c_4 \varepsilon/2}(\reg \cS), \]
    and $B_{c_3}^g(p) \subset U$. Note that 
    \[ \Vol_{\widehat{g}_\varepsilon}(B^g_{c_3}(p) \cap B^g_{c_4 \varepsilon}(\reg \cS) \setminus B^g_{c_4 \varepsilon/2}(\reg \cS)) \geq C_2 \varepsilon^2, \]
    with $C_2 = C_2(\cS, g_0, \Lambda) > 0$ since $\cS$ is an $(n-2)$-skeleton, i.e., it has codimension $\geq 2$. In sight of this, \eqref{eq:smooth.out.n2.faces.ii} implies
    \begin{align*}
    	& \int_M \phi(R(\widehat{g}_\varepsilon); \zeta(\cdot; \varepsilon))_+ \, d\Vol_{\widehat{g}_\varepsilon}- c_0^4 \int_M \phi(R(\widehat{g}_\varepsilon); \zeta(\cdot; \varepsilon))_{-} \, d\Vol_{\widehat{g}_\varepsilon}\\
    	& \qquad \geq 2 \cdot \varepsilon^{-2/q} \cdot C_2 \varepsilon^2 -\gamma C_1 \varepsilon^{2(q-1)/q},
    \end{align*}
    which can be made to be positive provided $\gamma$ is  sufficiently small depending on $\cS$, $g_0$, $\Lambda$.
\end{proof}

\section{Smoothing point singularities, $n = 3$} 

Our method for smoothing point singularities consists of:

\begin{enumerate}
	\item ``blowing up'' the singularity;
	\item excising the asymptotic end produced in the previous step by cutting along a particular minimal surface;
	\item ``filling in'' the holes created in the two previous steps with regions of positive scalar curvature.
\end{enumerate}

Step (1) is inspired from works of Schoen-Yau (e.g. \cite{SchoenYau79, SchoenYau81}). Steps (2) and (3) are inspired by constructions that feature in the second author's recent work with Pengzi Miao \cite{MantoulidisMiao17}. 

A new key necessary ingredient in this work is the following new excision lemma for asymptotic ends with weak regularity at infinity:

\begin{lemm} \label{lemm:bounded.minimal.surface.existence}
    Suppose $g$ is a $C^{2,\alpha}_{\loc}$ metric on $\RR^n \setminus B_1(\mathbf{0})$ with
    \begin{equation} \label{eq:minimal.surface.existence.0}
        \Lambda^{-1} \delta \leq g \leq \Lambda \delta \text{,}
    \end{equation}
    where $\delta$ is the standard flat metric on $\RR^n$ and $\alpha \in [0,1]$. If
    \[ \sC := \{ \Omega \subset \RR^n \text{ bounded open set containing } B_1(\mathbf{0}) \} \text{,} \]
    then $\inf \{ \cH^{n-1}_g(\partial \Omega) : \Omega \in \sC \}$ is attained by some $\Omega \subset B_R(\mathbf{0})$, $R = R(n, \Lambda)$.
\end{lemm}
\begin{proof}
    First, by a direct comparison argument, we have
    \begin{equation} \label{eq:bounded.minimal.surface.existence.i}
        \inf \{ \cH^{n-1}_g(\partial \Omega) : \Omega \in \sC \} \leq \cH^{n-1}_g(\partial B_1(\mathbf{0})) \leq c_1(n, \Lambda).
    \end{equation}
    Let $\{ \Omega_i \}_{i=1,2,\ldots} \subset \sC$ be a minimizing sequence of domains for the left hand side of \eqref{eq:bounded.minimal.surface.existence.i}, for each of which we denote
    \[ r_i := \inf \{ |\mathbf{x}| : \mathbf{x} \in \partial \Omega_i \}. \]
    Here, $|\mathbf{x}|$ denotes the Euclidean length of a position vector $\mathbf{x} \in \mathbf{R}^n$. By another direct comparison and the area formula on $(\RR^n, \delta)$,
    \[ \cH^{n-1}_g(\partial \Omega_i) \geq c_2'(n, \Lambda) \cH^{n-1}_\delta(\partial \Omega_i) \geq c_2'(n, \Lambda) \cH^{n-1}_\delta(\partial B_{r_i}(\mathbf{0})) = c_2''(n, \Lambda) r_i^{n-1}, \]
    which together with \eqref{eq:bounded.minimal.surface.existence.i} implies
    \begin{equation} \label{eq:bounded.minimal.surface.existence.ii}
        r_i \leq c_2(n, \Lambda) \text{ for all $i = 1, 2, \ldots$}
    \end{equation}
    For convenience, denote
    \[ r := \liminf_{i\to \infty} r_i \in [1, \infty), \]
    where the finiteness is a byproduct of  \eqref{eq:bounded.minimal.surface.existence.ii}. Pass to a subsequence that attains the $\liminf$. For that subsequence, let
    \[ R_i := \sup \{ |\mathbf{x}| : \mathbf{x} \in \partial \Omega_i \} \]
    and
    \[ R := \limsup_{i \to \infty} R_i \in [1, \infty]. \]
    Without loss of generality, $R > r$. We seek to estimate $R$ from above. Pass to yet another subsequence that attains the $\limsup$. 
    
    By a compactness argument, there will exist a closed $\Omega \subset \RR^n \setminus B_1(\mathbf{0})$ containing $\partial B_1(\mathbf{0})$ such that, by definition of $r$, $R$,
    \begin{equation} \label{eq:bounded.minimal.surface.existence.iii}
        \Sigma_t := \partial \Omega \cap \partial B_t(\mathbf{0}) \neq \emptyset \text{ for al } t \in (r, R),
    \end{equation}
    and, by \eqref{eq:bounded.minimal.surface.existence.i},
    \begin{equation} \label{eq:bounded.minimal.surface.existence.iv}
        \cH^{n-1}_g(\partial \Omega) \leq c_1(n, \Lambda).
    \end{equation}
    For each $t \in (r, R)$, let $h(t)$ denote the $\cH^{n-1}_g$-measure of the solution of the Plateau problem with prescribed boundary $\Sigma_t$; this is guaranteed to be nonzero by \eqref{eq:bounded.minimal.surface.existence.iii}. We do not concern ourselves with the technicalities behind the existence of a feasible minimizer in the Plateau problem--- we are content with the existence of a competitor with $\cH^{n-1}_g(\cdot) \leq 2h(t)$, which is guaranteed, for instance, by the deformation theorem.
    
    By \eqref{eq:minimal.surface.existence.0} and the isoperimetric inequality on $(\RR^n, \delta)$,
    \begin{equation} \label{eq:bounded.minimal.surface.existence.v}
        h(t)^{\tfrac{n-2}{n-1}} \leq c_3(n, \Lambda) \cH^{n-2}_g(\Sigma_t) \text{ for all } t \in (r, R).
    \end{equation}
    Moreover, we claim that
    \begin{equation} \label{eq:bounded.minimal.surface.existence.vi}
        2h(t) \geq \cH^{n-1}_g(\partial \Omega \setminus B_t(\mathbf{0})) \text{ for all } t \in (r, R);
    \end{equation}
    indeed, if this were false for some $t$, then a direct replacement could produce $\Omega' \in \sC$ with $\cH^{n-1}_g(\partial \Omega') < \cH^{n-1}_g(\partial \Omega) = c_1$, violating \eqref{eq:bounded.minimal.surface.existence.i}.
    
    The coarea formula, \eqref{eq:bounded.minimal.surface.existence.v}, and \eqref{eq:bounded.minimal.surface.existence.vi}, give:
    \begin{align*}
        2h(t) & \geq \cH^{n-1}_g(\partial \Omega \setminus  B_t(\mathbf{0})) \\
            & \geq \int_t^R \int_{\Sigma_s} |\nabla^T \dist_g(\mathbf{0}; \cdot)|^{-1} \, d\cH^{n-2}_g \, ds \\
            & \geq \int_t^R \cH^{n-2}_g(\Sigma_s) \, ds \\
            & \geq c_3^{-1} \int_t^R h(s)^{\tfrac{n-2}{n-1}} \, ds;
    \end{align*}
    the second equality follows from $|\nabla^T \dist(\mathbf{0}; \cdot)| \leq 1$. In other words, if $H(t)$ denotes the ultimate integral that appears above, we've shown that
    \[ |H'(t)|^{\tfrac{n-1}{n-2}} \geq (2c_3)^{-1} H(t) \text{ for all } t \in (r, R). \]
    In fact, since $H' \leq 0$, we get
    \[ -H'(t) \geq c_4(n, \Lambda) H(t)^{\tfrac{n-2}{n-1}} \text{ for all } t \in (r, R). \]
    Integrating, we find that there exists $R^\star = R^\star(n, \Lambda, r) \leq R^\star(n, \Lambda)$ such that $H(t) = 0$ for all $t \in [R^\star, R)$. This violates \eqref{eq:bounded.minimal.surface.existence.iii} unless $R \leq R^\star$, giving us an a priori bound on $R$.
    
    Finally, the finiteness of $R$ shows that the minimizing sequence $\Omega_i$ is trapped inside a fixed annulus, and the desired conclusion follows from standard compactness theorems in geometric measure theory.
\end{proof}

\begin{prop} \label{prop:smooth.out.vertices}
    Suppose $n = 3$, $\cS \subset M$ is finite, $\widetilde{g}$ is an $L^\infty(M) \cap C^{2,\alpha}_{\loc}(M \setminus \cS)$ metric, $\alpha \in (0,1)$, and $R(\widetilde{g}) > 0$ on $M \setminus \cS$. Then, there exists a $C^{2,\alpha}(M)$ metric $\overline{g}$ with $R(\overline{g}) > 0$ everywhere; i.e., $\sigma(M) > 0$.
\end{prop}
\begin{proof}
    We may assume, without loss of generality, that $\cS \neq \emptyset$, for else there is nothing to do. For notational simplicity we relabel $\widetilde{g}$ as $h$. Let $G$ denote the distributional solution of elliptic PDE
    \[ - 8 \Delta_h G + \phi(R(h); 1) G = \delta_{\cS} \text{ on } M, \]
    where $\delta_{\cS}$ denotes the Dirac delta measure on $\cS$, and $\phi$ is as in Section \ref{sec:smooth.out.n2.faces}. Since $h$ is uniformly Euclidean and $\phi(R(h);1)$ is bounded, we know that
    \begin{equation} \label{eq:greens.function.asymptotics.0}
		c_{0,G}^{-1} \dist_h(\cdot, \cS)^{-1} \leq G \leq c_{0,G} \dist_h(\cdot, \cS)^{-1} \text{ on } M \setminus \cS \text{,}
    \end{equation}
    for $c_{0,G} = c_{0,G}(M, h) > 0$, and, therefore,
    \begin{equation} \label{eq:greens.function.asymptotics}
		c_G^{-1} \dist_g(\cdot, \cS)^{-1} \leq G \leq c_G \dist_g(\cdot, \cS)^{-1} \text{ on } M \setminus \cS \text{.}
    \end{equation}
    We refer the reader to \cite{LittmanStampacchiaWeinberger63} for the existence and the aforementioned blow up rate of Greens functions in this setting.
    
	Consider, for small $\sigma > 0$, the conformal metric $h_\sigma = (1 + \sigma G)^4 h$ on $M \setminus  \cS$, which is $C^{2,\alpha}_{\loc}$, complete, noncompact, and whose scalar curvature satisfies
    \begin{align*}
    	R(h_\sigma) & = (1 + \sigma G)^{-5}(-8 \Delta_h (1 + \sigma G) + R(h) (1 + \sigma G)) \\
    		& = (1+\sigma G)^{-5} R(h) + (1+\sigma G)^{-5} \sigma (R(h) - \phi(R(h); 1)) G > 0.
    \end{align*}
    Fix a family of disjoint open neighborhoods of the points in $\cS$ (one for each point) labeled $\{ U_p \}_{p \in \cS}$, so that each $U_p \subset M$ is diffeomorphic to a 3-ball.
    
    \begin{claim}
        For every $p \in \cS$, there exists a diffeomorphism 
        \[ \Phi_p : \RR^3 \setminus B_1(0)  \stackrel{\approx}{\longrightarrow} U_p \setminus \{p\} \]
        and a constant $c_{p,\sigma} > 0$ such that
        \[ c_{p,\sigma}^{-1} \delta \leq \Phi_p{}^* h_\sigma \leq c_{p,\sigma} \delta \text{,} \]
        where $\delta$ denotes a flat metric on $\RR^3 \setminus B_1(0)$.
    \end{claim}
    \begin{proof}
        From the manifold's smooth structure, there exists a  diffeomorphism
        \[ \Psi_p : (B_1(0) \subset \RR^3) \stackrel{\approx}{\longrightarrow} (U_p \subset T), \]
        such that $\Psi_p(0) = p$. We can then define
        \[ \Phi_p := \Psi_p \circ \iota, \]
        where $\iota(x) = |x|^{-2} x$ is the inversion map on $\RR^3 \setminus \{0\}$. With this definition for $\Phi_p$, we see that
        \begin{align*}
        	\Phi_p{}^* h_\sigma & = \iota^* \Psi_p^* (1+\sigma G)^4 h \\
        		& = (1+\sigma (G \circ \Psi_p \circ \iota))^4 (\iota^* \Psi_p{}^* h).
        \end{align*}
        Next, note that $\Psi_p{}^* h$ is uniformly Euclidean on $B_1(0)$, and thus certainly on $B_1(0) \setminus \{0\}$. By the scaling nature of $\iota$, $\rho^4 (\Psi_p{}^* h)$ is uniformly Euclidean on $\RR^3 \setminus B_1(0)$, with $\rho$ denoting the standard radial polar coordinate on $\RR^3$. The result then follows from the asymptotics in \eqref{eq:greens.function.asymptotics}.
    \end{proof}
    
    \begin{claim}
        For every $p \in \cS$, and for every sufficiently small $\sigma > 0$, there exists a compact set $D_p \subset U_p$ whose boundary $\partial D_p$ consists of a stable minimal 2-spheres in $(M \setminus \cS, h_\sigma)$.
    \end{claim}
    \begin{proof}
        Lemma \ref{lemm:bounded.minimal.surface.existence} guarantee that for each $p \in \sing \cS$ there exists a compact surface $\Sigma = \Sigma_\sigma$ in $(U_p \setminus \{p\}, h_\sigma)$, with least $\cH^2_{h_\sigma}$-area among all surfaces in the same region that are homologous to $\partial U_p$. From standard regularity theory in geometric measure theory \cite{Simon83}, $\Sigma$ is regular and embedded away from $\partial U_p$. By a straightforward comparison argument and the smooth convergence $h_\sigma \to h$ away from $\cS$, we know that
        \begin{equation} \label{eq:area.minimization.decay}
        	\lim_{\sigma \to 0} \cH^2_{h_\sigma}(\Sigma_\sigma) = 0.
        \end{equation}
        Next, denote
        \[ W_{p,\tau} := \{ x \in U_p : \dist_g(x; \partial U_p) < \tau \}, \]
        where $\tau > 0$ is small. We will show that, for $\tau > 0$,
        \[ \Sigma \cap (W_{p,2\tau} \setminus W_{p,\tau}) = \emptyset, \]
        as long as $\sigma > 0$ is sufficiently small (depending on $\tau$).
        
        Assume, by way of contradiction, that there exists a sequence $\sigma_j \downarrow 0$ such that the corresponding area-minimizing surfaces $\Sigma_j = \Sigma_{\sigma_j}$ are such that $\Sigma_j \cap (W_{p,2\tau} \setminus W_{p,\tau}) \neq \emptyset$ for all $j = 1, 2, \ldots$ For each $j$, pick $p_j \in \Sigma_j \cap (W_{p,2\tau} \setminus W_{p,\tau})$, and denote $T_j$ the connected component of $\Sigma_j$ in $W_{p,2\tau}$ containing the point $p_j$. By the local monotonicity formula in small regions of Riemannian manifolds, and the fact that $h_{\sigma_j} \to h$ smoothly away from $\sing \cS$, we know that
        \begin{equation} \label{eq:area.minimization.decay.obstruction}
        	\liminf_j \cH^2_{h_{\sigma_j}}(T_j) > 0.
        \end{equation}
        However, \eqref{eq:area.minimization.decay.obstruction} contradicts \eqref{eq:area.minimization.decay}.

		Thus, $\Sigma \cap (W_{p,2\tau} \setminus W_{p,\tau}) = \emptyset$ as long as $\sigma$ is sufficiently close to zero. This implies that
		\begin{enumerate}
			\item $\Sigma' \subset \overline{W}_{p,\tau}$, or
			\item $\Sigma' \cap W_{p,2\tau} = \emptyset$,
		\end{enumerate}
		for every connected component $\Sigma' \subset \Sigma$.  Case (1) cannot occur for arbitrarily small $\sigma > 0$: there is a positive lower $\cH^2_h$-area bound for all non-null-homologous surfaces in $\overline{W}_{p,\tau}$, in violation of \eqref{eq:area.minimization.decay}. Thus, (2) holds for all connected components $\Sigma' \subset \Sigma$. Therefore,
		\[ \Sigma \cap W_{p,2\tau} = \emptyset \implies \Sigma \cap \partial U_p = \emptyset, \]
		so $\Sigma$ is a regular embedded minimal surface in $(U_p \setminus \{p\}, h_\sigma)$. Using $R(h_\sigma) > 0$ and the main theorem of \cite{FischerColbrieSchoen80}, we conclude that $\Sigma$ consists of stable minimal 2-spheres. The fact that $\Sigma$ bounds a compact region follows from topological considerations.
    \end{proof}
    
    Fix $\sigma > 0$ small enough so that the previous claim applies for all $p \in \cS$, where the corresponding $\cH^2_{h_\sigma}$-area minimizing surfaces are $\Sigma_p$, $p \in \cS$. Denote $\Sigma := \cup_{p \in \cS} \Sigma_p$.
    
    Combining \cite[Lemma 2.2.1]{Mantoulidis17} and \cite[Corollary 2.2.13]{Mantoulidis17}, we deduce that there exists a smooth manifold $N^3$, diffeomorphic to $M^3$, and a metric $\overline{h}$ on $N$ which is uniformly $C^2$ on the complement of the image $\Sigma' \subset N$ of $\Sigma \subset M$, Lipschitz across $\Sigma'$, and such that $\Sigma'$ is minimal from both sides. Since the mean curvatures of $\Sigma'$ from both sides agree, we may use the mollification procedure in \cite[Proposition 4.1]{Miao02} to smooth out $\overline{h}$ to $\widetilde{h}$; the result follows by applying the conformal transformation in Lemma \ref{lemm:lq.positive.scalar.curvature} to $\widetilde{h}$, with $\chi = R(\widetilde{h})$.
\end{proof}

\section{Proof of main theorems}

\begin{proof}[Proof of Theorem {\ref{theo:main.highD}}]
	Notice that $\sing \cS = \emptyset$, so, by Proposition \ref{prop:smooth.out.n2.faces}, either
	\begin{enumerate}
		\item $g$ is $C^2$ and $R(g) \equiv 0$ on $M$, or
		\item there exists a $C^2$ metric $\widetilde{g}$ on $M$ with $R(\widetilde{g}) > 0$ everywhere.
	\end{enumerate}
  The latter contradicts the assumption $\sigma(M) \leq 0$, so the prior must be true. In that case, the result follows from Theorem \ref{theo:sigma.obstruction.and.rigidity}.
\end{proof}

\begin{proof}[Proof of Theorem {\ref{theo:main.3D}}]
	By Proposition \ref{prop:smooth.out.n2.faces}, either
	\begin{enumerate}
		\item $g$ is $L^\infty(M) \cap C^{2,\alpha}_{\loc}(M \setminus \sing \cS)$ with $R(g) \equiv 0$, or
		\item there exists an $L^\infty(M) \cap C^{2,\alpha}_{\loc}(M \setminus \sing \cS)$ metric $\widetilde{g}$ with
			\[ R(\widetilde{g}) > 0 \text{ on } M \setminus \sing \cS. \]
	\end{enumerate}
	Proposition \ref{prop:smooth.out.vertices} rules out the second case when $\sigma(M) \leq 0$, so the first case holds.
	
	\begin{claim}
		$\Ric(g) \equiv 0$ on $M \setminus \sing \cS$.
	\end{claim}
	\begin{proof}[Proof of claim]
		We argue by contradiction. Suppose that $\Ric(g) \neq 0$ at some $p \in M \setminus \sing \cS$. Let $U$ be some small smooth open neighborhood of $p$ such that $U \subset \subset M \setminus \sing \cS$; note that $g|_U \in C^{2,\alpha}(U)$.
		
		Consider the Banach manifold
		\[ \mathcal{M}^{2,\alpha}_g(U) := \{ \text{metrics } g' \in C^{2,\alpha}(U) : g - g' \equiv 0 \text{ on } \partial U \}, \]
		where $g - g' \equiv 0$ on $\partial U$ is to be interpreted as the equality of tensors on $\overline{U}$ pointwise on the subset $\partial U \subset \overline{U}$; i.e., we aren't pulling back to $\partial U$. The scalar curvature functional
		\[ R : \mathcal{M}_g^{2,\alpha}(U) \to C^{0,\alpha}(U) \]
		is a $C^1$ Banach map, with Fr\'echet derivative $\delta R(g') : T_{g'} \mathcal{M}_g^{2,\alpha}(U) \to C^{0,\alpha}(U)$ known to be given by
		\[ \delta R(g')\{h\} = - \Delta_{g'} \Tr_{g'} h + \operatorname{div}_{g'} \operatorname{div}_{g'} h - \langle h, \Ric(g') \rangle_{g'}, \]
		for all $g' \in \mathcal{M}_g^{2,\alpha}(U)$, $h \in T_{g'} \mathcal{M}_g^{2,\alpha}(U) \cong (T_0^{2,\alpha} \otimes T_0^{2,\alpha})^*(U)$. Here, $T_{g'} \mathcal{M}_g^{2,\alpha}(U)$ denotes the tangent space at $g'$ to the Banach manifold $\cM_g^{2,\alpha}(U)$, and $T_0^{2,\alpha}(U)$ denotes the space of contravariant $C^{2,\alpha}$ tensors that vanish on $\partial U$.
		
		Fix some $h \in T_g \mathcal{M}_g^{2,\alpha}(U)$. Define $\gamma : [0,\delta) \subset \mathcal{M}_g^{2,\alpha}(U)$ to be a $C^1$ curve with $\gamma(0) = g$, $\gamma'(0) = -h$. By definition of Fr\'echet derivatives, the fact $R(g) \equiv 0$, and the trivial continuous embedding $C^{0,\alpha}(U) \hookrightarrow L^\infty(U)$, we have
		\begin{equation} \label{eq:main.3D.frechet.derivative}
		    \lim_{t \downarrow 0} \frac{\Vert R(\gamma(t)) - \Delta_g \Tr_g th + \operatorname{div}_g \operatorname{div}_g th - \langle th, \Ric(g) \rangle_g \Vert_{L^\infty(U)}}{t} = 0.
		\end{equation}
		In particular, by observing that the Fr\'echet derivative contains two divergence terms that integrate to zero with respect to $d\Vol_g$, we have
		\begin{align}
		    & \lim_{t \downarrow 0} \frac{1}{t} \int_U (R(\gamma(t)) - \langle th, \Ric(g) \rangle_g) \, d\Vol_g \nonumber \\
		    & \qquad = \lim_{t \downarrow 0} \frac{1}{t} \int_U (R(\gamma(t)) - \Delta_g \Tr_g th + \operatorname{div}_g \operatorname{div}_g th - \langle th, \Ric(g) \rangle_g) \, d\Vol_g \nonumber \\
		    & \qquad = 0, \label{eq:main.3D.integral.R}
		\end{align}
		where the last equality follows from \eqref{eq:main.3D.frechet.derivative}. 
		
		Consider the map $t\in(-\ep,\ep)\mapsto \lambda_1(t)=\lambda_1(-\frac{4(n-1)}{n-2}\Delta_{g_t}+R(g_t))$. Since $g_t$ is an smooth (in $t$) family of $C^{2,\alpha}$ metrics, we know that $t \mapsto \lambda_1(t)$ is $C^1$; the corresponding first eigenfunctions, $u_t$, normalized to have $\Vert u_t \Vert_{L^2(M,g_t)} = 1$, form a $C^1$ path in $W^{1,2}(M)$. (See, e.g., \cite[Lemma A.1]{MantoulidisSchoen15}). Notice that $u_0$ is a constant, $R_0 \equiv 0$ on $M \setminus \sing \cS$, and $\lambda_1(0) = 0$. Observe that
		\begin{align*}
		    \lambda'(0)
		        &=\td{}{t}\bigg\vert_{t=0} \int_M \left[\frac{4(n-1)}{n-2}|\nabla_{g_t}u_t|^2+R(g_t)u_t^2\right]d\Vol_{g_t} \\
		        &=\int_M \td{}{t}\bigg\vert_{t=0}R(g_t)d\Vol_{g}\\
		        &=\int_M \bangle{h,\Ric(g)}d\Vol_g.
		\end{align*}
		where we have used the fact that the only nonzero contribution of the derivative is from $\td{}{t}R(g_t)$ and \eqref{eq:main.3D.integral.R}.
		
		Suppose, now, that $(\Ric(g))_\sigma$ denotes a (tensorial) mollification of $\Ric(g)$ away from $p$, such that
		\begin{equation} \label{eq:main.3D.mollification.Linfty}
		    \lim_{\sigma \to 0} \Vert (\Ric(g))_\sigma - \Ric(g) \Vert_{L^\infty(U)} = 0.
		\end{equation}
		If $\xi : M \to [0,1]$ is a smooth cutoff function such that $\xi(p) = 1$ and $\spt \xi \subset U$, then \eqref{eq:main.3D.mollification.Linfty} implies
		\begin{equation} \label{eq:main.3D.mollification.L2}
		    \lim_{\sigma \to 0} \langle \xi (\Ric(g))_\sigma, \Ric(g) \rangle_{L^2(U,g)} > 0.
		\end{equation}
		Together, \eqref{eq:main.3D.integral.R},  \eqref{eq:main.3D.mollification.L2} imply that for all sufficiently small $\sigma > 0$,
		\[\lambda(t)>0,\]
		for all $t \in (0, t_0(\sigma))$, when $h = \xi (\Ric(g))_\sigma \in T_g \mathcal{M}_g^{2,\alpha}(U)$; in fact, since $\gamma(t)$, $g$ are all uniformly equivalent for small $\sigma$, $t$, we have
		\[ \int_U R(\gamma(t)) \, d\Vol_{\gamma(t)} > 0 \text{ for } t \in (0, t_0(\sigma)). \]
		Now fix a small $\sigma$. This implies that for any $t \in (0, t_0(\sigma))$, $\widetilde{g}_t=u_t^{\frac{4}{n-2}}g_t$ is a $L^\infty(M) \cap C^{2,\alpha}_{\loc}(M \setminus \sing \cS)$ metric with \emph{positive} scalar curvature on its regular part; this contradicts Proposition \ref{prop:smooth.out.vertices} when $\sigma(M) \leq 0$.
		
	\end{proof}
	
	Given the above, all that remains to be checked is that $g$ is smooth across $\sing \cS$. This follows (when $n = 3$) from the main theorem of Smith-Yang \cite{SmithYang92} on the removability of isolated singularities of Einstein metrics.
\end{proof}

We now turn our attention onto asymptotically flat manifolds and prove Theorem \ref{theo:PMT.highD}, \ref{theo:PMT.3D}. The idea is to take the smoothed metric $g_\ep$ in Lemma \ref{lemm:smooth.out.edge.singularities} and apply a conformal deformation to $g_\ep$ with small change of the ADM mass. Assume $(M^n,g)$ is an asymptotically flat manifold, $\cS\subset M$ is a compact nondegenerate $(n-2)$-skeleton which is $\eta$-regular along $\reg \cS$.

\begin{proof}[Proof of Theorem \ref{theo:PMT.highD}]
    Notice that $\sing \cS=\emptyset$. By Lemma \ref{lemm:smooth.out.edge.singularities}, for every $\gamma>0$, there exists constant $\ep_1$ such that for every $\ep\in (0,\ep_1]$, there is a metric $\hat{g_\ep}$ on $M$ such that:
    \begin{enumerate}
        \item $\widehat{g}_\ep$ is $C^2(M)$;
        \item $\widehat{g}_\ep=g$ on $M\setminus B_\ep^g(\reg \cS)$;
        \item $\|R(\widehat{g}_\ep)_{-}\|_{L^{\frac{n}{2}+\delta}(M,g)}\le \gamma$;
    \end{enumerate}
    
    By the maximum principle and the Poincar\'e-Sobolev inequality, we conclude that the elliptic boundary value system
    \[ \begin{Bmatrix}
            \Delta_{\widehat{g}_\ep}u_\ep+c_n R(\widehat{g}_\ep)_{-}u_\ep = 0 \\
            \lim_{x\rightarrow\infty}u_\ep = 1 \\
            u_\ep = 0 \text{ on } \partial M
        \end{Bmatrix} \]
    has a unique solution $u_\ep$, and $0<u_\ep<1$. This follows as in \cite{SchoenYau79pmt}. The same argument as in \cite[Proposition 4.1]{Miao02}, moreover, shows that
    \[\lim_{\ep\rightarrow 0}\|u_\ep-1\|_{L^\infty(M)}=0,\quad \|u_\ep\|_{C^{2,\alpha}(K)}\le C_K,\]
    for each compact set $K \subset M \setminus \cS$, where $C_K=C_K(g,\cS,\Lambda,K)$. 
    
    Now define $\widetilde{g}_\ep=u_\ep^{\frac{4}{n-2}}\widehat{g}_\ep$. By the choice of $u_\ep$, $R(\widetilde{g_\ep})\ge 0$ everywhere. We then apply the argument of \cite[Lemma 4.2]{Miao02} and conclude that
    \[ m_{ADM}(g) = \lim_{\ep \to 0} m_{ADM}(\widetilde{g}_\ep), \]
    which is $\geq 0$ by the smooth positive mass theorem \cite{SchoenYau79pmt, SchoenYau81, SchoenYau17}.
    
    If the cone angle along $\cS$ is not identically $2\pi$, then Lemma \ref{lemm:smooth.out.edge.singularities} additionally gives the following concentration behavior of scalar curvature:
    \[R(\widehat{g_\ep})\ge C_1 \ep^{-2} \text{ on }B_{c_2\ep}^g(\cS)\setminus B_{c_3\ep}^g(\cS),\]
    where $C_1=C_1(\cS,g,\Lambda)$, $c_j=c_j(\cS,g,\Lambda)$, $j=2,3$.
    Then \cite[Proposition 4.2]{Miao02} implies that
    \[\liminf_{\ep \to 0} m_{ADM}(\widetilde{g_\ep})>0,\]
    and hence $m_{ADM}(g)>0$. Hence if $m_{ADM}(g)=0$, then $g$ is smooth across $\cS$, and therefore the rigidity conclusion of the smooth positive mass theorem in \cite{SchoenYau17} implies that $g$ is flat everywhere.
\end{proof}

\begin{proof}[Proof of Theorem \ref{theo:PMT.3D}]
    Take $\widehat{g}_\ep$ as in Lemma \ref{lemm:smooth.out.edge.singularities}. By the maximum principle and the Poincar\'e-Sobolev inequality, the weak $W^{1,2}_{\loc}(M)$ solution of 
    \[ \begin{Bmatrix}
            \Delta_{\widehat{g}_\ep}u_\ep+c_n R(\widehat{g}_\ep)_{-}u_\ep = 0 \\
            \lim_{x\rightarrow\infty}u_\ep = 1 \\
            u_\ep = 0 \text{ on } \partial M
        \end{Bmatrix} \]
    exists and satisfies $0<u_\ep<1$. By standard elliptic theory and De Giorgi-Nash-Moser theory, $u_\ep \in C^{2,\alpha}_{\loc}(M\setminus \sing \cS)\cap C^{0,\theta}(M)$, for some $\theta \in (0,1)$. Moreover,
    \[ \inf u_\ep \geq c_1 \sup u_\ep = c_1 = c_1(g, \cS, \Lambda), \]
    by Moser's Harnack inequality.
    
    The metric $\widetilde{g}_\ep=u_\ep^4 \widehat{g}_\ep$ is asymptotically flat with only isolated singularities and nonnegative scalar curvature away from $\sing \cS$. Let $G(\cdot,\sing \cS)$ be the Green's function of $\Delta_{\widetilde{g_\ep}}-\frac{1}{8}\phi(R(\widetilde{g_\ep};1))$ with poles at $\sing\cS$, and which decays to zero at infinity. Apply the excision lemma (Lemma \ref{lemm:bounded.minimal.surface.existence}) to the blown up metric
    \[h_\ep=(1+\sigma(\ep)G)^4 \widetilde{g}_\ep \]
    on $M\setminus \sing \cS$, where $\sigma(\ep)$ is the constant that appears in the proof of Proposition \ref{prop:smooth.out.vertices}, and $\lim_{\ep \to 0} \sigma(\ep) = 0$. Excise $(M,h_\ep)$ along each area minimizing two-sphere in the asymptotically Euclidean end in $(M,h_\ep)$. 
    
    Notice that, since $\lim_{\ep \to 0} \sigma(\ep) = 0$,
    \[ m_{ADM}(g) = \lim_{\ep\rightarrow 0} m_{ADM}(\widetilde{g}_\ep) = \lim_{\ep\rightarrow 0} m_{ADM}(h_\ep) \]
    on each asymptotically flat end of $M$ (recall that we've excised one end). By the smooth positive mass theorem \cite{SchoenYau79pmt, SchoenYau81}, $m_{ADM}(h_\ep)\ge 0$. Therefore $m_{ADM}(g)\ge 0$. (To see the above limits on the ADM masses, we first notice that the metrics $h_\ep$ and $\widetilde{g}_\ep$ differ by a factor which converges to zero in $C^2$, as $x\rightarrow \infty$ and $\ep\rightarrow 0$. Hence from the definition of ADM mass, $\lim_{\ep\rightarrow 0}m(h_\ep)=\lim_{\ep\rightarrow 0}m(\widetilde{g}_\ep)$. To see that $\lim_{\ep}(\widetilde{g}_\ep)=m(g)$, we just apply \cite[Lemma 4.2]{Miao02} again on the family of conformal factors $u_\ep$.)
    
    Now we conclude the rigidity case. Assume $m_{ADM}(g)=0$. If $R(g)$ is not identically zero on $\reg\cS$, or the cone angle along $\reg \cS$ is not identically $2\pi$, then a similar concentration behavior as in the proof of Theorem \ref{theo:PMT.highD}, combined with \cite[Proposition 4.2]{Miao02}, show that
    \[\liminf_{\ep \to 0}m_{ADM}(h_\ep)>0;\]
    this would contradict our rigidity assumption. Therefore $g$ is scalar flat on $M\setminus \cS$, and is $C^{2,\alpha}$ across $\reg \cS$ locally away from $\sing \cS$. Now we prove $\Ric(g)=0$ away from $\sing \cS$. Consider the metrics $g_t=g-th$, where $h$ is a $C^{2,\alpha}$ symmetric $(0,2)$ tensor, compactly supported away from $\sing\cS$. Let $u_t$ be the weak solution to
    \[ \begin{Bmatrix}
            \Delta_{g_t}u_t+c_n R(g_t)_{-}u_t = 0 \\
            \lim_{x\rightarrow\infty}u_t = 1 \\
            u_t = 0 \text{ on } \partial M
        \end{Bmatrix} \]
    Then (see, e.g., the proof of Theorem \ref{theo:main.3D} for details) the metric $\widehat{g}_t=u_t^4 g_t$ has zero scalar curvature and isolated uniformly Euclidean point singularities. Therefore $m_{ADM}(\widehat{g}_t)\ge 0$ by the positive mass theorem for isolated $L^\infty$ singularities established just above. On the other hand, by a similar calculation as in \cite{SchoenYau79pmt}, we see that
    \[\td{}{t}\bigg\vert_{t=0}m_{ADM}(M, \widehat{g}_t)=C_1(n)\int_M \bangle{\Ric(g),h}.\]
    Now if $\Ric(g)\ne 0$ in an open neighborhood of $M\setminus \sing \cS$, we may pick $h=\xi (\Ric(g))_\sigma$, where $\xi$ is a function compactly supported in $U$, $(\Ric(g))_\sigma$ is a $C^{2,\alpha}$ mollification of $\Ric(g)$, and make $m_{ADM}'(0)\ne 0$. (See the proof of Theorem \ref{theo:main.3D}.) This is a contradiction to the positive mass theorem for isolated $L^\infty$ singularities, which would imply that $m_{ADM}(0) = 0$ is a global minimum of $t \mapsto m_{ADM}(M, \widehat{g}_t)$.
    
    Finally, being in $n=3$, we conclude that $g$ is smooth and flat across $\sing \cS$ by the removable singularity theorem \cite{SmithYang92} of Einstein metrics.
\end{proof}

\section{Examples, counterexamples, remarks} \label{sec:examples.counterexamples}

\subsection{Codimension-1 singularities and mean curvature} \label{subsec:codimension.1.appendix} Question \ref{quest:singular.psc.rigidity} is \emph{true} when $\cS \subset M$ is a closed, embedded, two-sided submanifold, and:
\begin{enumerate}
    \item $g$ is smooth up to $\cS$ from both sides, 
    \item $g$ induces the same metric $g_{\cS}$ on $\cS$ from both sides, and
    \item the sum of mean curvatures of $\cS$ computed with respect to the two unit normals as outward unit normals is nonnegative.
\end{enumerate}

We'd like to point out that condition (3) is imperative, as the following counterexample clearly shows: take a flat $n$-torus $\RR^n / \ZZ^n$, remove a small geodesic ball, and replace it with a constant curvature half-sphere of the same radius. Indeed, here the sum of mean curvatures is negative (one negative, the other zero), and the resulting metric $g$ does not have a removable singular set.

For some intuition on (3), one may use the first variation of mean curvature along a geodesic foliation of $M$ about $\cS$, and the Gauss equation on $\cS$, to see that
\begin{equation} \label{eq:codim.1.scalar.curv.formula}
    R(g)|_{\cS} = R(g_{\cS}) - \left[ \frac{d}{dt} H_t \right]_{t=0} -  |A_{\cS}|^2 - H_{\cS}^2.
\end{equation}
Heuristically, a positive sum of mean curvatures contributes a distributionally positive component to the scalar curvature $R(g)$ evaluated at $\cS$. 

\subsection{Codimension-2 singularities and cone angles} \label{subsec:codimension.2.appendix}

Allowing edge metrics with cone angles larger than $2\pi$ invalidates Question \ref{quest:singular.psc.rigidity}. We illustrate this here with a counterexample:

The example in Figure \ref{Example.angle.big} inspired by \cite[Example 5.6-B']{Gromov83}. For each integer $g\ge 2$, we describe a flat metric on a genus $g$ Riemann surface with isolated conical points with cone angle $> 2\pi$.

Take a planar graph $G$ with two nodes, $p$ and $q$, and $g+1$ edges. The graph separates the plane into $g+1$ connected components. Excise one disk from each bounded face and, and the exterior of a disk from the unbounded face, as in Figure \ref{Example.angle.big}.

\begin{figure}[h!]
    \centering
    \includegraphics[scale=0.25]{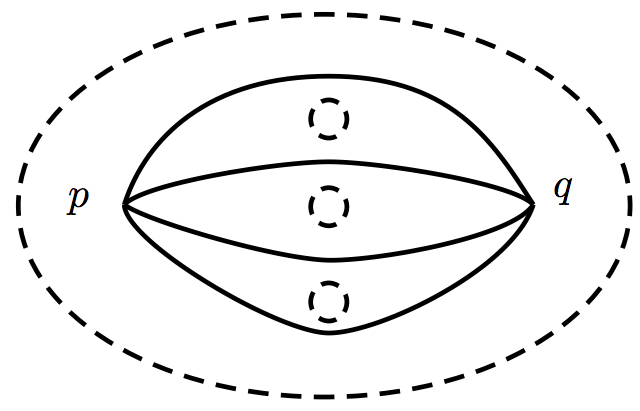}
    \caption{$g=3$ construction.}
    \label{Example.angle.big}
\end{figure}

Each face component is diffeomorphic to $\SS^1\times [0,1]$. Endow each face with a flat product metric via this diffeomorphism. Note that the metric now is smooth away from $p,q$, and is conical at $p$, $q$, with cone angle $(g+1)\pi$. The $g+1$ boundary components of this manifold, namely the $\SS^1\times \{1\}$'s, are totally geodesic. Take the doubling of this manifold across its boundary to obtain a genus-$g$ surface, $\Sigma$. Now $\Sigma$ has a smooth flat metric with four isolated conical singularities, each of which has cone angle $(g+1)\pi$.

For $n\ge 3$, consider the manifold $M=\Sigma\times (\SS^1)^{n-2} $ with the product metric. It's easy to see that this metric is an edge metric, flat on its regular part, with singularities with cone angles $(g+1)\pi$. However, since $M$ trivially carries a smooth metric with nonpositive sectional curvature, its $\sigma$-invariant satisfies $\sigma(M)\le 0$ by \cite[Corollary A]{GromovLawson84}. 

\subsection{Codimension-3 singularities that are not uniformly Euclidean.} \label{subsec:codimension.3.appendix} The ``uniformly Euclidean'' condition (i.e., that the metric be $L^\infty$) is imperative for Conjecture \ref{conj:schoen} to hold true. Indeed, if $g$ were allowed to blow up, then the doubled Riemannian Schwarzschild metric
\[ g = \left( 1 + \frac{m}{2r} \right)^4 \delta, \; m > 0, \]
on $\RR^3 \setminus \{0\}$ would be a counterexample: it can be viewed as a non-Einstein, scalar-flat metric on a twice-punctured 3-sphere.
    
Likewise, if $g$ were not bounded from below, then the negative-mass Riemannian Schwarzschild metric
\[ g = \left( 1 + \frac{m}{2r} \right)^4 \delta, \; m < 0, \]
on $\RR^3 \setminus \overline{B}_{-m/2}(0)$ would yield a counterexample: it can be conformally truncated near infinity to match Euclidean space, where we then identity the opposide faces of a large cube to yield a topologically smooth 3-torus with positive scalar curvature, which would be a counterexample since tori are known to have $\sigma(T^3) = 0$. See \cite[Section 6]{Lohkamp99} for more details.

\subsection{Examples of edge metrics} \label{subsec:source.edge.metrics}

Orbifold metrics provide an important source of edge metrics (with angle $< 2\pi$); they can be obtained as the quotient metric under a $\ZZ_k$ isometry group with an $(n-2)$-dimensional fixed submanifold. 

Generally speaking, the scalar curvature geometry of orbifolds can be substantially different from that of manifolds. For instance, Viaclovsky \cite{Viaclovsky10} showed that the Yamabe problem of finding constant scalar curvature metrics is \emph{not} generally solvable on orbifolds. (On manifolds, the problem was shown to be completely solvable in \cite{Trudinger68, Aubin76, Schoen84, SchoenYau88}.) Theorem \ref{theo:main.highD} nevertheless confirms that edge-type orbifold singularities along a codimension two submanifold cannot go so far as to change the Yamabe type from nonpositive to positive.

\subsection{Gromov's polyhedral comparison} \label{subsec:gromov.comparison}

We focus here on the aspect of \cite{Gromov14} that relates most closely with our work. This regards metric aspects of cubes in three-dimensional manifolds with nonnegative scalar curvature.

Let $(M^3, g)$ be a polyhedron of cube type, $M \approx [0,1]^3 \subset \RR^3$, with faces $F_j$. Let $\measuredangle_{ij} (M, g)$ denote the (possibly nonconstant) dihedral angle between two adjacent faces $F_i$ and $F_j$. Per Gromov's polyhedral comparison theory for nonnegative scalar curvature, $(M^3, g)$ cannot simultaneously satisfy:
\begin{enumerate}
    \item the scalar curvature is nonnegative, $R(g) \geq 0$;
    \item each $F_i$ is mean convex;
    \item for all two adjacent faces $F_i$, $F_j$, $\measuredangle_{ij} (M, g) < \tfrac{\pi}{2}$.
\end{enumerate}
Gromov's crucial observation is to argue, by contradiction, that such an arrangement would stand in violation of $\sigma(\TT^3) = 0$. Indeed, he proposes the following elegant construction: ``double'' $M$ three times ---across the front, right, and bottom faces--- and then identify all newly created isometric faces opposite of each other to obtain a torus, $\TT^3$. Then $g$ lifts to a metric $\widetilde{g}$ on $\mathbf{T}^3$ whose singular set stratifies as $F^2 \cup E^1 \cup V^0$, and:
\begin{enumerate}
    \item $\widetilde{g}$ is smooth on both sides of $F^2$, induces the same smooth metric on $F^2$, and the sum of mean curvatures of $F^2$ computed with respect to the two unit normals as outward unit normals is positive (see Section \ref{subsec:codimension.1.appendix});
    \item $\widetilde{g}$ is an edge metric along $E^1$ with angles less than $2\pi$; 
    \item $\widetilde{g}$ is uniformly Euclidean across $V^0$;
\end{enumerate}
Each stratum is independently compatible with a ``weak'' notion of nonnegative scalar curvature: see \cite{Miao02, ShiTam16} for codimension one, and Theorem \ref{theo:main.3D} for codimensions two and three. However, tori shouldn't carry such metrics.

Using different methods altogether, the first-named author confirmed in \cite{Li} that such arrangements don't exist and, moreover, that if one allows weakly mean curvex faces and non-obtuse dihedral angles, then such arrangements are rigid: they are necessarily rectangular domains in $\RR^3$.

\subsection{Sormani-Wenger intrinsic flat distance} \label{subsec:intrinsic.flat.distance}

For the reader's convenience, we recall the Sormani-Wenger definition:

\begin{defi}[{\cite[Definition 1.1]{SormaniWenger11}}] \label{defi:intrinsic.flat.convergence}
    Let $(M^n_1, g_1)$, $(M^n_2, g_2)$ be two closed Riemannian manifolds. Their intrinsic flat distance is defined as
    \[ d_{\mathcal{F}}((M_1, g_1), (M_2, g_2)) := \inf \{ d_F^Z((\varphi_1)_{\#} T_1, (\varphi_2)_{\#} T_2) : Z, \varphi_1, \varphi_2 \}; \]
    the infimum is taken over all complete metric spaces $(Z, d)$ and all possible isometric embeddings $\varphi_i$, $i = 1, 2$, of the metric spaces induced by $(M_i, g_i)$ into $(Z, d)$. The $T_i$, $i = 1, 2$ denote the integral $n$-currents $T_i(\omega) := \int_{M_i} \omega$, $(\varphi_i)_{\#} T_i$ denote their pushforwards to $Z$, and $d_F^Z$ denotes the Ambrosio-Kirchheim metric space flat norm \cite{AmbrosioKirchheim00}:
    \[ d_F^Z(S, T) := \inf \{ \mathbf{M}(U) + \mathbf{M}(V) : S-T = U + \partial V \}; \]
    this infimum is taken over integral $n$-currents $U$ and integral $(n+1)$-currents $V$ in $(Z, d)$.
\end{defi}

In this paper we have shown that the following families of singular Riemannian manifolds $(M^n, g)$ in Theorems \ref{theo:main.highD}, \ref{theo:main.3D} will either have:
\begin{enumerate}
    \item $\sigma(M) \leq 0$ and be everywhere smooth and Ricci-flat to begin with; or
    \item $\sigma(M) > 0$ and carry smooth metrics of positive scalar curvature.
\end{enumerate}

We conjecture that, in the second case, the desingularizations we have set up in this paper give rise to $d_{\mathcal{F}}$-Cauchy sequences of smooth closed PSC manifolds, which, moreover, recover $(M^n, g)$ as a metric $d_{\mathcal{F}}$-limit.

A more ambitious conjecture, that appears out of reach with today's state of the art, is to show that $(M^n, g)$, $n \geq 4$, with singular sets of codimension $\geq 3$ and positive scalar curvature everywhere else, arise as metric $d_{\mathcal{F}}$-limits of smooth closed PSC manifolds; cf. Conjecture \ref{conj:schoen}.

\bibliography{main} 
\bibliographystyle{amsalpha}

\end{document}